\documentclass[10pt]{amsart}

 \usepackage{amssymb, latexsym}

\usepackage{hyperref}

\hypersetup{
unicode=true, 
pdftoolbar=true, 
pdfmenubar=true, 
pdffitwindow=true, 
pdftitle={ spectral theory for the weak decay of muons in a uniform magnetic field} 
pdfauthor={jean-claude guillot}, 
pdfsubject={article}, 
pdfnewwindow=true, 
pdfkeywords={polonium, P6}, 
colorlinks=true, 
linkcolor=blue, 
citecolor=blue, 
filecolor=blue, 
urlcolor=blue 
}
\hypersetup{
linktocpage=true
}

\DeclareMathOperator*{\slim}{s-lim}

\DeclareMathOperator*{\produ}{\prod}

\def\C{{\mathbb{C}}} 
\def\N{{\mathbb{N}}} 
\def\R{{\mathbb{R}}} 

\def\vp{\varphi}

\def\1{{\mathbf{1}}}

\renewcommand\d{\mathrm{d}}

\def\1{1\hskip-0.09cm{\rm l}}

\newtheorem{theorem}{Theorem}[section]

\newtheorem{proposition}[theorem]{Proposition}

\newtheorem{hypothesis}[theorem]{Hypothesis}

\newcommand{\arxivlink}[1]{\href{http://arxiv.org/abs/#1}{arXiv:#1}}

\begin{document}
\title[Weak Interactionsin an uniform magnetic field]{ spectral theory for the weak decay of muons in a uniform magnetic field}

\author[J.-C. Guillot]{Jean-Claude Guillot}

\address[J.-C. Guillot]{CMAP, Ecole polytechnique, CNRS, Institut Polytechnique de Paris, 91128 Palaiseau, France}

\email{Jean-Claude.Guillot@polytechnique.edu}



\subjclass[2010]{81V15, 81V10, 81Q10}

\keywords{Weak Decay of muons,Fermi's theory, Uniform Magnetic Field, Spectral Theory,}

\maketitle



\begin{abstract}
In this article we consider a mathematical model for the weak decay of muons in a uniform magnetic field according to the Fermi theory of weak interactions with V-A coupling. With this model we associate a Hamiltonian with cutoffs in an appropriate Fock space. No infrared regularization is assumed. The Hamiltonian is self-adjoint and has a unique ground state. We specify the essential spectrum and prove the existence of asymptotic fields from which we determine the absolutely continuous spectrum. The coupling constant is supposed sufficiently small.
\end{abstract}

\maketitle
\tableofcontents
\section{Introduction.}
\setcounter{equation}{0}

In this paper we consider a mathematical model for the weak decay of muons into electrons, neutrinos and antineutrinos in a uniform magnetic field according to the Fermi theory with V-A (Vector-Axial Vector)coupling,
\begin{equation}\label{1}
   \mu_{-}\rightarrow e_{-} + \overline{\nu_{e}} + \nu_{\mu}
\end{equation}
\begin{equation}\label{2}
  \mu_{+}\rightarrow e_{+} + \nu_{e} + \overline{\nu}_{\mu}
\end{equation}

\eqref{2} is the charge conjugation of \eqref{1}.

This is a part of a program devoted to the study of mathematical models for the weak interactions as patterned according to the Fermi theory and the Standard model in Quantum Field Theory. See \cite{GM}.

In this paper we restrict ourselves to the study of the decay of the muon $\mu_{-}$ whose electric charge is the charge of the electron \eqref{1}. The study of the decay of the antiparticle $\mu_{+}$, whose charge is positive, \eqref{2}  is quite similar and we omit it.

In \cite{G17} we have studied the spectral theory of the Hamiltonian associated to the inverse $\beta$ decay in a uniform magnetic field. We proved the existence and uniqueness of a ground state and we specify the essential spectrum and the spectrum for a small coupling constant and without any low-energy regularization.

In this paper we consider the weak decay of muons into electrons, neutrinos  associated with muons and  antineutrinos associated with electrons in a uniform magnetic field according to the Fermi theory with V-A coupling. Hence we neglect the small mass of neutrinos and antineutrinos and we define a total Hamiltonian H acting in an appropriate Fock space involving three fermionic massive particles-the electrons, the muons and the antimuons - and two fermionic massless particles- the neutrinos and the antineutrinos associated to the muons and the electrons respectively.In order to obtain a well-defined operator, we approximate the physical kernels of the interaction Hamiltonian by square integrable functions and we introduce high-energy cutoffs. We do not need to impose any low-energy regularization in this work but the coupling constant is supposed sufficiently small.

We give a precise definition of the Hamiltonian as a self-adjoint operator in the appropriate Fock space and by adapting the methods used in \cite{G17}we first state  that H has a unique ground state and we specify the essential spectrum for sufficiently small values of the coupling constant.

In this paper, our main result is the location of the absolutely continuous spectrum of H. For that we follow the first step of the approach to scattering theory in establishing, for each involved particle,the existence and basic properties of the asymptotic creation and annihilation operators for time $t$ going to $\pm\infty$. We then have a natural definition of unitary wave operators with the right intertwining property from which we deduce the absolutely continuous spectrum of H. Scattering theory for models in Quantum Field Theory without any external field has been considered by many authors. See, among others,\cite{Amm, BFS, DG, FGS,H}, \cite{HK1, HK2, HK3,HK4, HK5, HS, KM1, KM2, KM3}, \cite{T1,T2,BDH}and references therein. A part of the techniques used in this paper are adapted from the ones developed in these references. Note that the asymptotic completeness of the wave operators is an open problem in the case of the weak interactions in the background of a uniform magnetic field. See \cite{AF} for a study of scattering theory for a mathematical model of the weak interactions without any external field.

In some parts of our presentation we will only give the statement of theorems referring otherwise to some references.

The paper is organized as follows. In the second section we define the regularized self-adjoint Hamiltonian associated to \eqref{1}. In the third section we consider the existence of a unique ground state and we specify the essential spectrum of H. In the fourth section we carefully prove the existence of asymptotic limits, when time t goes to $\pm\infty$, of the creation and annihilation operators of each involved particle, we define a unitary wave operator and we prove that it satisfies the right intertwining property with the hamiltonian and we deduce the absolutely continuous spectrum of H. In Appendices A and B we recall the Dirac quantized fields associated to the muon and the electron in a uniform magnetic external field together with the Dirac quantized free fields associated to the neutrino and the antineutrino.

\section{The Hamiltonian.}\mbox{}
\setcounter{equation}{0}

In the Fermi theory the decay of the muon $\mu$ is described by the following four fermions effective Hamiltonian for the interaction in the Schr\"{o}dinger picture ( see \cite{GM}, \cite{GS} and \cite{WII}):
\begin{equation}\label{2.1}
  \begin{split}
  &H_{int} =\\&\frac{G_F}{\sqrt{2}} \int \d^3 \!x\,\big( \overline{\Psi}_{\nu_\mu}(x) \gamma^\alpha
  (1-\gamma_5)\Psi_{\mu}(x)\big)\big( \overline{\Psi}_{e}(x) \gamma_{\alpha}(1-\gamma_5)\Psi_{\nu_e}(x)\big)\\
  & + \frac{G_F}{\sqrt{2}} \int \d^3 \!x\,\big( \overline{\Psi}_{\nu_e}(x) \gamma_{\alpha}(1-\gamma_5)\Psi_{e}(x)\big)
  \big( \overline{\Psi}_{\mu}(x) \gamma^\alpha(1-\gamma_5)\Psi_{\nu_\mu}(x)\big)
 \end{split}
\end{equation}
Here $\gamma^{\alpha}$, $\alpha=0,1,2,3$ and $\gamma_5$ are the Dirac matrices in the standard representation. $\Psi_{(.)}(x)$ and $\overline{\Psi}_{(.)}(x)$ are the quantized Dirac
fields for $e$, $\mu$, $\nu_\mu$ and $\nu_e$. $\overline{\Psi}_{(.)}(x)= \Psi_{(.)}(x)^\dag \gamma^{0}$. $G_{F}$ is the Fermi coupling constant with $G_{F}\simeq 1.16639(2)\times10^{-5}GeV^{-2}$. See \cite{B}.

We recall that $m_e<m_\mu$. $\nu_\mu$ and $\nu_e$ are massless particles.

\subsection{The free Hamiltonian.}\mbox{}

Throughout this work notations are introduced in appendices A and B.

Let

\begin{equation}\label{2.2}
 \mathfrak{F}=\mathfrak{F}_e\otimes\mathfrak{F}_\mu\otimes\mathfrak{F}_\mu\otimes\mathfrak{F}_{\nu_\mu}\otimes\mathfrak{F}_{\overline{\nu}_e}
\end{equation}

Let
\begin{equation}\label{2.3}
  \begin{split}
    &\omega(\xi_1)=E_n^{(e)}(p^3)\quad for\  \xi_1= (s,n,p^1,p^3) \\
    &\omega(\xi_2)=E_n^{(\mu)}(p^3)\quad for\  \xi_2= (s,n,p^1,p^3) \\
    &\omega(\xi_3)=|\textbf{p}|\quad for\  \xi_3= (\textbf{p} , \frac{1}{2})\\
    &\omega(\xi_4)= |\textbf{p}|\quad for\  \xi_4= (\textbf{p} , -\frac{1}{2}).
    \end{split}
\end{equation}

Let $H_D^{(e)}$ \big( resp.$H_D^{(\mu_-)}$,$H_D^{(\mu_+)}$, and $H_D^{(\nu)}$ \big) be the Dirac Hamiltonian for the electron \big(resp.the muon, the antimuon and the neutrino \big).

The quantization of  $H_D^{(e)}$, denoted by $H_{0,D}^{(e)}$ and acting on $\mathfrak{F}_{e}$, is given by
\begin{equation}\label{2.4}
  H_{0,D}^{(e)}==\int \omega(\xi_1) b^{*}_{+}(\xi_{1}) b_{+}(\xi_{1})\d \xi_{1}
\end{equation}

 Likewise the quantization of $H_D^{(\mu_-)}$,$H_D^{(\mu_+)}$,$H_D^{(\overline{\nu}_e)}$ and $H_D^{(\nu_\mu)}$, denoted by $H_{0,D}^{(\mu)}$, $H_{0,D}^{(\overline{\nu}_e)}$ and $H_{0,D}^{(\nu_\mu)}$ respectively, acting on $\mathfrak{F}_\mu$, $\mathfrak{F}_{\overline{\nu}_e}$ and  $\mathfrak{F}_{\nu_\mu}$ respectively, is given by
\begin{equation}\label{2.5}
  \begin{split}
    &H_{0,D}^{(\mu_-)}=\int \omega(\xi_2) b^{*}_{+}(\xi_{2}) b_{+}(\xi_{2})\d \xi_{2}\, \\
    &H_{0,D}^{(\mu_+)}=\int \omega(\xi_2) b^{*}_{-}(\xi_{2}) b_{-}(\xi_{2})\d \xi_{2}\, \\
    &H_{0,D}^{(\overline{\nu}_e)}=\int \omega(\xi_3) b^{*}_{-}(\xi_{3}) b_{-}(\xi_{3})\d \xi_{3} \\
    &H_{0,D}^{(\nu_\mu)}= \int \omega(\xi_4) b^{*}_{+}(\xi_{4}) b_{+}(\xi_{4})\d \xi_{4}.
    \end{split}
\end{equation}

We set $H_{0,D}^{(\mu)}$ = $H_{0,D}^{(\mu_-)}\otimes \1 $ +$ \1 \otimes H_{0,D}^{(\mu_+)}$.  $H_{0,D}^{(\mu)}$ is defined on $\mathfrak{F}_\mu \otimes\mathfrak{F}_\mu$.

For each Fock space $\mathfrak{F}_{.}$ let $\mathfrak{D}^{(.)}$ denote the set of vectors $\Phi\in\mathfrak{F}_{.}$ for which each component $\Phi^{(r)}$  is smooth and has a compact support and $\Phi^{(r)} = 0$ for all but finitely many $r$. Then $H_{0,D}^{(.)}$  is
well-defined on the dense subset $\mathfrak{D}^{(.)}$  and it is essentially self-adjoint on $\mathfrak{D}^{(.)}$ . The
self-adjoint extension will be denoted by the same symbol $H_{0,D}^{(.)}$  with domain $D(H_{0,D}^{(.)})$).

The spectrum of $H_{0,D}^{(e)}$ in $\mathfrak{F}_{(e)}$ is given by
\begin{equation}\label{2.6}
  \mathrm{spec}(H_{0,D}^{(e)})=\{0\}\cup [m_e,\infty)
\end{equation}
$\{0\}$ is a simple eigenvalue whose the associated eigenvector is the vacuum in  $\mathfrak{F}^{(e)}$ denoted by $\Omega^{(e)}$. $[m_e,\infty)$ is the absolutely continuous spectrum of $H_{0,D}^{(e)}$.

Likewise the spectra of $H_{0,D}^{(\mu)}$, $H_{0,D}^{(\overline{\nu}_e)}$ and $H_{0,D}^{(\nu_\mu)}$ in $\mathfrak{F}_{\mu}\otimes \mathfrak{F}_{\mu}$, $\mathfrak{F}_{\overline{\nu_e}}$ and $\mathfrak{F}_{\nu_\mu}$ respectively  are given by
\begin{equation}\label{2.7}
  \begin{split}
    &\mathrm{spec}(H_{0,D}^{(\mu)})=\{0\}\cup [m_{\mu},\infty) \\
    &\mathrm{spec}(H_{0,D}^{(\overline{\nu_e})})=[0,\infty)\\
    &\mathrm{spec}(H_{0,D}^{(\nu_\mu)})=[0,\infty).
  \end{split}
\end{equation}
 $\Omega^{(\mu)}$, $\Omega^{(\overline{\nu}_e)}$ and  $\Omega^{(\nu_\mu)}$ are the associated vacua in  $\mathfrak{F}_{\mu}\otimes \mathfrak{F}_{\mu}$, $\mathfrak{F}_{\overline{\nu_e}}$ and  $\mathfrak{F}_{\nu_\mu}$ respectively and are the associated eigenvectors of  $H_{0,D}^{(\mu)}$,$H_{0,D}^{(\overline{\nu}_e)}$ and $H_{0,D}^{(\nu_\mu)}$ respectively for the eigenvalue $\{0\}$.

The vacuum in $\mathfrak{F}$, denoted by $\Omega$, is then given by
\begin{equation}\label{2.8}
  \Omega=\Omega^{(e)} \otimes \Omega^{(\mu)} \otimes \Omega^{(\overline{\nu_e})} \otimes \Omega^{(\nu_\mu)}
\end{equation}

The free Hamiltonian for the model, denoted by $H_0$ and acting in $\mathfrak{F}$, is now given by
\begin{equation}\label{2.9}
  \begin{split}
  &H_0=H_{0,D}^{(e)} \otimes \1 \otimes \1 \otimes \1 \otimes \1 + \1 \otimes H_{0,D}^{(\mu)}\otimes \1 \otimes\1 \otimes 1 \\
  & + \1 \otimes \1 \otimes \1 \otimes  H_{0,D}^{(\overline{\nu_e})}\otimes \1 + \1 \otimes \1 \otimes \1\otimes \1 \otimes H_{0,D}^{(\nu_\mu)}.
  \end{split}
\end{equation}

$H_0$ is essentially self-adjoint on $\mathfrak{D}=\mathfrak{D}^{(e)}\widehat\otimes\mathfrak{D}^{(\mu)}\widehat\otimes\mathfrak{D}^{(\overline{\nu_e})}\widehat\otimes\mathfrak{D}^{(\nu_\mu)}$.

Here $\widehat\otimes$ is the algebraic tensor product.

$\mathrm{spec}(H_0)=[0,\infty)$ and $\Omega$ is the eigenvector associated with the simple eigenvalue $\{0\}$ of $H_0$.

Let $S^{(e)}$ be the set of the thresholds of $H_{0,D}^{(e)}$:
$$S^{(e)}=\big(s^{(e)}_n; n\in \N \big)$$
with $s^{(e)}_n=\sqrt{m_{e}^2 + 2neB}$.

Likewise let  $S^{(\mu)}$ be the set of the thresholds of $H_{0,D}^{(\mu)}$:
$$S^{(\mu)}=\big(s^{(\mu)}_n; n\in \N \big)$$
with $s^{(\mu)}_n=\sqrt{m_{\mu}^2 + 2neB}$.

Then
\begin{equation}\label{2.10}
  \mathfrak{S}= S^{(e)}\cup S^{(\mu)}
\underline{}\end{equation}
is the set of the thresholds of $H_0$.

Throughout this work any finite tensor product of annihilation or creation operators associated with the involved particles will be denoted for shortness by the usual product of the operators (see e.g \eqref{2.13} and \eqref{2.14}).
\subsection{The Interaction.}\mbox{}

Similarly to \cite{G17},\cite{BG},\cite{ABFG},\cite{G15},\cite{BFG1},\cite{BFG2}and \cite{BFG3} in order to get well-defined operators on $\mathfrak{F}$, we have to substitute
smoother kernels $F(\xi_2,\xi_4)$ and $G(\xi_1,\xi_3)$ for the $\delta$-distribution associated with \eqref{2.1}( conservation of momenta) and for introducing ultraviolet cutoffs.

Let
\begin{equation}\label{2.11}
  \textbf{r}= \textbf{p}_3 + \textbf{p}_4
\end{equation}

We get a new operator denoted by $H_I$ and defined as follows
\begin{equation}\label{2.12}
  H_I= H_I^1 + (H_I^1)^{*} + H_I^{2} + (H_I^{2})^{*}
\end{equation}

Here

\begin{equation}\label{2.13}
\begin{split}
  &H^{(1)}_{I}= \int \d \xi_1 \d \xi_2 \d \xi_3 \d \xi_4\,\Big(\int \d x^2  \mathrm{e}^{-ix^2r^2} \\
  &\big( \overline{U}^{(\nu_\mu)}(\xi_4) \gamma^\alpha(1-\gamma_5)U^{(\mu)}(x^2,\xi_2)\big)\big( \overline{U}^{(e)}(x^2,\xi_1) \gamma_{\alpha}(1-\gamma_5)U^{(\nu)}(\xi_4)\big)\Big)\\
  &F(\xi_{2}, \xi_4)\,G(\xi_1,\xi_3)b_{+}^{*}(\xi_4)b_{+}^{*}(\xi_1)b_{-}^{*}(\xi_3)b_{+}(\xi_2).
  \end{split}
\end{equation}

and

\begin{equation}\label{2.14}
  \begin{split}
  &H^{(2)}_{I}= \int \d \xi_1 \d \xi_2 \d \xi_3 \d \xi_4\,\Big(\int \d x^2  \mathrm{e}^{-ix^2r^2} \\
  &\big( \overline{U}^{(\nu_\mu)}(\xi_4) \gamma_{\alpha}(1-\gamma_5)W^{(\mu)}(x^2,\xi_2)\big)\big( \overline{U}^{(e)}(x^2,\xi_1) \gamma^\alpha(1-\gamma_5)W^{(\overline{\nu_e})}(\xi_3)\big)\Big)\\
  &F(\xi_{2}, \xi_{4})\,G(\xi_1,\xi_3)b_{+}^{*}(\xi_4)b_{-}^{*}(\xi_2)b^{*}_{+}(\xi_1)b^{*}_{-}(\xi_3).
  \end{split}
\end{equation}

$H^{(1)}_{I}$ describes the decay of the muon and $H^{(2)}_{I}$ is responsible for the fact that the bare vacuum will not be an eigenvector of the total Hamiltonian as expected from physics.

We now introduce the following assumptions on the kernels $F(\xi_{2}, \xi_{4})$ and $G(\xi_1,\xi_3)$ in order to get well-defined Hamiltonians in $\mathfrak{F}$.

\begin{hypothesis}\label{hypothesis:2.1}
\begin{equation}\label{2.15}
\begin{split}
&F(\xi_2, \xi_4)\in
 L^2(\Gamma_1\times \R^3) \\
 &G^(\xi_1, \xi_3)\in
 L^2(\Gamma_1\times \R^3)
\end{split}
\end{equation}
\end{hypothesis}
These assumptions will be needed throughout the paper.

By \eqref{2.12}-\eqref{2.15} $H_I$ is well defined as a  sesquilinear form on $\mathfrak{D}$ and one can construct a closed operator associated with this form.

The total Hamiltonian is thus

\begin{equation}\label{2.16}
  H= H_0 + gH_I, \quad g>0.
\end{equation}
$g$ is the coupling constant that we suppose non-negative for simplicity. The conclusions below are not affected if $g\in\R$.

The self-adjointness of $H$ is established by the next theorem.

Let

\begin{equation}\label{2.17}
  \begin{split}
     C= &\|\gamma^0\gamma_\alpha (1-\gamma_5)\|_{\C^4}\|\gamma^0\gamma^\alpha (1-\gamma_5)\|_{\C^4}. \\
    \frac{1}{M}=&\frac{1}{m_e} + \frac{1}{m_\mu}
   \end{split}
\end{equation}

For $\phi \in D(H_0)$ we have

 \begin{equation}\label{2.18}
  \begin{split}
   \|H_I\phi\|  &  \\
     \leq & 2C\|F(.,.)\|_{L^2(\Gamma_1\times\R^3)}\|G(.,.)\|_{L^2(\Gamma_1\times\R^3)}\big(\frac{2}{M}\|H_0\phi\| + \|\phi\| \big).
  \end{split}
\end{equation}

\eqref{2.18} follows from standard estimates of creation and annihilation operators in Fock space (the $N_\tau$ estimates, see \cite{GJ}). Details can be found in \cite[proposition 3.7]{BDG}.

\begin{theorem}\label{2.2}(Self-adjointness).
Let $g_0>0$ be such that
\begin{equation}\label{2.19}
4g_0\frac{C}{M}\|F(.,.)\|_{L^2(\Gamma_1\times\R^3)}\|G(.,.)\|_{L^2(\Gamma_1\times\R^3)}<1.
\end{equation}
Then for any g such that $g\leq g_0$ H is self-adjoint in $\mathfrak{F}$ with domain $\mathcal{D}(H)= \mathcal{D}(H_0)$. Moreover any core for $H_0$ is a core for $H$.
\end{theorem}
By \eqref{2.18} and \eqref{2.19} the proof of the self-adjointness of $H$ follows from the Kato-Rellich theorem.

$\sigma(H)$ stands for the spectrum and $\sigma_{ess}(H)$ denotes the essential spectrum. We have
\begin{theorem}\label{2.3}(The essential spectrum and the spectrum)
Setting
\begin{equation*}
    E=\inf\sigma(H)
\end{equation*}
we have for every $g\leq g_0$
\begin{equation*}
    \sigma(H)=\sigma_{\mathrm{ess}}(H)= [E,\infty)
\end{equation*}
with $E\leq 0$ .
\end{theorem}

In order to prove the theorem \ref{2.3} we easily adapt to our case the proof given in \cite{BFG3} (see also \cite{G17}, \cite{A} and \cite{T3}). The mathematical model considered in \cite{BFG3} involves also one neutrino and one antineutrino. We omit the details.

\section{Existence of a unique ground state.}

In the sequel we shall make some of the following additional assumptions on the kernels $F(\xi_2,\xi_4)$ and $G(\xi_1,\xi_3)$.

\begin{hypothesis}\label{hypothesis:3.1} There exists a constant  $K(F,G)>0$
such that for $\sigma>0$
\begin{equation*}
\begin{split}
  \mbox{(i)}\quad &
\int_{\Gamma_1\times\R^3}
 \frac{|
 F(\xi_2,\xi_4)|^2}{|\textbf{p}_{4}|^2}
 \d\xi_1 \d\xi_4  <\infty .
\end{split}
\begin{split}
 \mbox{(ii)}\quad &
 \int_{\Gamma_1\times\R^3}
 \frac{|
 G(\xi_1,\xi_3)|^2}{|\textbf{p}_{3}|^2}
 \d\xi_1 \d\xi_3  <\infty .
\end{split}
\end{equation*}
\begin{equation*}
\begin{split}
\mbox{(iii)}\quad &
 \left(\int_{\Gamma_1\times\{|\textbf{p}_{4}|\leq \sigma\}}
 |F(\xi_2,\xi_4)|^2
 \d\xi_2 \d\xi_4 \right)^{\frac{1}{2}} \leq K(F,G) \sigma.
\end{split}
\end{equation*}
\begin{equation*}
\begin{split}
\mbox{(iv)}\quad &
 \left(\int_{\Gamma_1\times\{|\textbf{p}_{3}|\leq \sigma\}}
 |G(\xi_1,\xi_3)|^2
 \d\xi_1 \d\xi_3 \right)^{\frac{1}{2}} \leq K(F,G) \sigma.
\end{split}
\end{equation*}
\end{hypothesis}

We then have

\begin{theorem}\label{3.1}
Assume that the kernels $F(.,.)$ and $G(.,.)$ satisfy Hypothesis 2.1 and 3.1. Then there exists $g_1\in (0,g_0]$ such that $H$ has a unique ground state for $g \leq g_1$.
\end{theorem}

In order to prove theorem 3.1 it suffices to mimick the proofs given in  \cite{G17},\cite{ABFG}and \cite{BFG3}. We omit the details.

In \cite{AFG} fermionic hamiltonian models are considered without any external field. Without any restriction on the strengh of the interaction a self-adjoint hamiltonian is defined for which the existence of a ground state is proved. Such a result is an open problem in the case of magnetic fermionic models.

\section{The absolutely continuous spectrum.} \mbox{}
\setcounter{equation}{0}

As stated in the introduction , in order to specify the absolutely continuous spectrum of H, we follow the first step of the approach to scattering theory in establishing, for each involved particle, the existence and basic properties of the asymptotic creation and annihilation operators for time $t$ going to $\pm\infty$. The existence of a ground state is quite fundamental in order to get a Fock subrepresentation of the asymptotic canonical anticommutation relations from which we localize the absolutely continuous spectrum of H.

\subsection{Asymptotic fields.} \mbox{}
Let
\begin{equation}\label{4.1}
  \begin{split}
  b_{1,+,t}^{\sharp}(f_1) &= e^{itH}e^{-itH_0} b_{1,+}^{\sharp}(f_1)e^{itH_0}e^{-itH} \\
  b_{2,\pm,t}^{\sharp}(f_2) &= e^{itH}e^{-itH_0} b_{2,\pm}^{\sharp}(f_2)e^{itH_0}e^{-itH} \\
  b_{3,-,t}^{\sharp}(f_3)&= e^{itH}e^{-itH_0} b_{3,-}^{\sharp}(f_3)e^{itH_0}e^{-itH} \\
  b_{4,+,t}^{\sharp}(f_4)&= e^{itH}e^{-itH_0} b_{4,+}^{\sharp}(f_4)e^{itH_0}e^{-itH}.
  \end{split}
\end{equation}
where, for $i=1,2$, $f_i \in L^2(\Gamma_1)$ and, for $j=3,4$, $f_j \in L^2(\R^3)$.

The strong limits of $b_{.,t}^{\sharp}(.)$ when the time t goes to $\pm\infty$ for models in Quantum Field Theory have been considered for fermions and bosons by \cite{KM1}-\cite{KM3} and \cite{HK1}-\cite{HK5} and , more recently, by  \cite{Amm},\cite{DG},\cite{H},\cite{T1},\cite{T2} and \cite{BDH} and references therein.

In the sequel we shall make some of the following additional assumptions on the kernels $F(\xi_2,\xi_4)$ and $G(\xi_1,\xi_3)$.

\begin{hypothesis}\label{hypothesis: 4.1}
\begin{equation*}
  \begin{split}
  \mbox{(i)}\quad &
  \int\left|
 \frac{\partial F}
 {\partial p_{\mu}^3}(\xi_2,\xi_4)\right|^2
 \d\xi_2 \d\xi_4  < \infty\ ,
 \int\left|
 \frac{\partial G}
 {\partial p_{e}^3}(\xi_1,\xi_3)\right|^2
 \d\xi_1 \d\xi_3  < \infty\ .\\
  \mbox{(ii)}\quad &
  \int\left|
\left(\frac{\partial}
 {\partial p_{\mu}^3}\right)^2F(\xi_2,\xi_4)\right|^2
 \d\xi_2 \d\xi_4  < \infty\ ,
 \int\left|
\left(\frac{\partial}
 {\partial p_{e}^3}\right)^2G(\xi_1,\xi_3)\right|^2
 \d\xi_1 \d\xi_3  < \infty\ .
 \end{split}
\end{equation*}
\end{hypothesis}

\begin{hypothesis}\label{hypothesis: 4.2}
\begin{equation*}
  \begin{split}
  \mbox{(i)}\quad &
  \int\left|
 \nabla_{p_\mu}F
 (\xi_2,\xi_4)\right|^2
 \d\xi_2 \d\xi_4  < \infty\ ,
 \int\left|
 \nabla_{p_e}G
 (\xi_1,\xi_3)\right|^2
 \d\xi_1 \d\xi_3  < \infty\ . \\
 \mbox{(ii)}\quad &
  \int\left|\frac{\partial^2 F}
 {\partial p_{\nu_\mu}^1 \partial p_{\nu_\mu}^3}(\xi_2,\xi_4,)\right|^2
 \d\xi_1 \d\xi_2 < \infty\ ,
 \int\left|\frac{\partial^2 G}
 {\partial p_{\overline{\nu}_e}^1 \partial p_{\overline{\nu}_e}^3}(\xi_1,\xi_3,)\right|^2
 \d\xi_1 \d\xi_3 < \infty\ .
\end{split}
\end{equation*}
\end{hypothesis}

We then have

 \begin{theorem}\label{4.3}
Suppose Hypothesis 2.1-Hypothesis 4.2 and $g \leq g_0$. Let $f_1,f_2 \in L^2(\Gamma_1)$ and $f_3,f_4 \in L^2(\R^3)$ . Then the following asymptotic fields
\begin{equation}\label{4.2}
 \begin{split}
  b_{1,+,\pm\infty}^{\sharp}(f_1) &:=\slim_{t\to\pm\infty} b_{1,+,t}^{\sharp}(f_1)  \\
  b_{2, \pm,\pm\infty}^{\sharp}(f_2)&:=\slim_{t\to\pm\infty} b_{2,\pm,t}^{\sharp}(f_2)  \\
  b_{3,-,\pm\infty}^{\sharp}(f_3)&:=\slim_{t\to\pm\infty}b_{3,-,t}^{\sharp}(f_3)  \\
  b_{4,+,\pm\infty}^{\sharp}(f_+)&:=\slim_{t\to\pm\infty} b_{4,+,t}^{\sharp}(f_4) .
 \end{split}
\end{equation}
exist.
\end{theorem}

\begin{proof}
The norms of the $b_{.,.,t}(f_.)'s$ are uniformly bounded with respect to t. Hence, in order to prove theorem 4.3 it suffices to prove the existence of the strong limits on $D(H)=D(H_0)$ with smooth $f_.$.

\noindent \textbf{Strong limits of $b_{1,+,t}^{\sharp}(f_1)$ and $b_{2,\pm,t}^{\sharp}(f_2)$ }.

Let
\begin{multline}\label{4.3}
\mathfrak{D}= \{f \in l^2(\Gamma_1)| f(s,n,.,.) \in C_0^\infty(\R^2 \setminus \{0\}) \mbox{ for all s and n} \ , \mbox{and}   \\
  f(.,n,.,.)=0 \mbox{ for all but finitely many n} \}.
\end{multline}

 Let $f_1,f_2\in \mathfrak{D}$ . According to \cite[lemma1]{HK1} we have
 \begin{equation}\label{4.4}
   b_{1,+}^{\sharp}(f_1)D(H) \subset D(H)\quad \mbox{and}\quad b_{2,\pm}^{\sharp}(f_2)D(H)\subset D(H).
 \end{equation}

 Moreover we have
 \begin{equation}\label{4.5}
 \begin{split}
 e^{itH_0}b_{1,+}^{\sharp}(f_1)e^{-itH_0}\Psi=&b_{1,+}^{\sharp}(e^{itE^e}f_1)\Psi, \\
 e^{itH_0}b_{2,\pm}^{\sharp}(f_2)e^{-itH_0}\Psi=&b_{2,\pm}^{\sharp}(e^{itE^\mu}f_2)\Psi.
 \end{split}
 \end{equation}
where $\Psi \in D(H)$.

Let us first prove the existence of $ b_{1,+,\pm\infty}^{\sharp}(f_1)$.

Let $\Psi \in D(H)$ and $f_{1,t}(\xi_1)=(e^{-itE^e}f_1)(\xi_1)$. By \eqref{4.4},\eqref{4.5} and the strong differentiability of $e^{itH}$ we get

\begin{equation}\label{4.6}
b_{1,+,T}(f_1)\Psi- b_{1,+,T_{0}}(f_1)\Psi=ig \int_{T_{0}}^{T} e^{itH}[ H_I,b_{1,+}(f_{1,t})] e^{-itH}\Psi dt.
 \end{equation}

By using the usual canpnicol anticommutation relations (CAR) (see\eqref{A.4}) we easily get for all $\Psi \in D(H)$
\begin{equation}\label{4.7}
\begin{split}
&\left[H_I^{(1)},b_{1,+}(f_{1,t})\right]\Psi=
 \int \d \xi_1 \d \xi_2 \d \xi_3 \d \xi_4 \,\Big(\int \d x^2 \mathrm{e}^{-ix^2r^2} \\
&\big( \overline{U}^{(\nu_\mu)}(\xi_4) \gamma_{\alpha}(1-\gamma_5)U^{(\mu)}(x^2,\xi_2)\big)
\big(\overline{U}^{(e)}(x^2,\xi_1) \gamma^\alpha(1-\gamma_5)W^{(\overline{\nu}_{e})}(\xi_3)\big)\Big)\\                                                &\overline{f_{1,t}(\xi_1)}F(\xi_{2}, \xi_{4})G(\xi_1,\xi_3)b_{+}^{*}(\xi_4)b_{-}^{*}(\xi_3)b_{+}(\xi_2)\Psi.
\end{split}
\end{equation}

\begin{equation}\label{4.8}
\begin{split}
&\left[H_I^{(2)},b_{1,+}(f_{1,t})\right]\Psi=
- \int \d \xi_1 \d \xi_2 \d \xi_3 \d \xi_4 \Big(\int \d x^2  \mathrm{e}^{-ix^2r^2} \\
&\big( \overline{U}^{(\nu_\mu)}(\xi_4) \gamma^\alpha(1-\gamma_5)W^{(\mu)}(x^2,\xi_2)\big)
\big(\overline{U}^{(e)}(x^2,\xi_1)\gamma_{\alpha}(1-\gamma_5)W^{(\overline{\nu}_{e})}(\xi_3)\big)\\                 &\overline{f_{1,t}(\xi_1)}F(\xi_{2},\xi_{4})G(\xi_1,\xi_3)\Big)
b_{+}^{*}(\xi_4)b^{*}_{-}(\xi_2)b^{*}_{-}(\xi_3)\Psi.
\end{split}
\end{equation}

\begin{equation}\label{4.9}
  \left[(H_I^{(1)})^*,b_{1,+}(f_{1,t})\right]\Psi=\left[(H_I^{(2)})^*,b_{1,+}(f_{1,t})\right]\Psi=0
\end{equation}

where $\overline{U}= U^\dagger\gamma^0$.

Similarly we get

\begin{equation}\label{4.10}
 b^*_{1,+,T}(f_1)\Psi- b^*_{1,+,T_{0}}(f_1)\Psi=ig \int_{T_{0}}^{T} e^{itH} [ H_I,b^*_{1,+}(f_{1,t})] e^{-itH}\Psi dt
\end{equation}

with
\begin{equation}\label{4.11}
  \left[H_I^{(1)},b^*_{1,+}(f_{1,t})\right]\Psi=\left[H_I^{(2)},b^*_{1,+}(f_{1,t})\right]\Psi=0
\end{equation}

and

\begin{equation}\label{4.12}
\begin{split}
&\left[(H_I^{(1)})^*,b^*_{1,+}(f_{1,t})\right]\Psi=
- \int \d \xi_1 \d \xi_2 \d \xi_3 \d\xi_4 \,\Big(\int \d x^2 \mathrm{e}^{ix^2r^2} \\
&\big( \overline{W}^{(\overline{\nu}_e)}(\xi_3) \gamma^\alpha(1-\gamma_5)U^{(e)}(x^2,\xi_1)\big)
\big(\overline{U}^{(\mu)}(x^2,\xi_2) \gamma_\alpha(1-\gamma_5)U^{(\nu_\mu)}(\xi_4)\big)\Big)\\                                                                        &\overline{F(\xi_{2}, \xi_{4})}\, \overline{G(\xi_{1}, \xi_{3})} f_{1,t}(\xi_1)
b_{+}^{*}(\xi_2)b_{-}(\xi_3)b_{+}(\xi_4)\Psi.
\end{split}
\end{equation}

\begin{equation}\label{4.13}
  \begin{split}
  &\left[(H_I^{(2)})^{*},b^*_{1,+}(f_{1,t})\right]\Psi=
   \int \d \xi_1 \d \xi_2 \d \xi_3 \d \xi_4\,\Big(\int \d x^2  \mathrm{e}^{ix^2r^2} \\
  &\big(\overline{W}^{(\overline{\nu}_e)}(\xi_3) \gamma^\alpha(1-\gamma_5)U^{(e)}(x^2,\xi_1)\big)
  \big( W^{(\mu)}(x^2,\xi_2) \gamma_\alpha(1-\gamma_5) U^{(\nu_\mu)}(\xi_4)\big)\Big)\\                                                                         &\overline{F(\xi_{2}, \xi_{4})}\, \overline{G(\xi_1,\xi_3)} f_{1,t}(\xi_1)
  b_{-}(\xi_3)b_{-}(\xi_2)b_{+}(\xi_4)\Psi.
  \end{split}
\end{equation}

By \eqref{4.6} and \eqref{4.10}, in order to prove the existence of $ b_{1,+,\pm\infty}^{\sharp}(f_1)$, we have to estimate $$e^{itH}[ H_I,b_{1,+}(f_{1,t})] e^{-itH}\Psi$$
and $$e^{itH}[ H_I,b^*_{1,+}(f_{1,t})] e^{-itH}\Psi$$
for large $|t|$.

By \eqref{B.5}, the $N_\tau$ estimates (see \cite{GJ} and \cite[Proposition 3.7]{BDG}), \eqref{A.8}, \eqref{A.11} and \eqref{A.13} we get
\begin{equation}\label{4.14}
  \begin{split}
  &\left\|e^{itH}[ H_I^{(1)},b_{1,+}(f_{1,t})] e^{-itH}\Psi\right\|\leq   \\
  &C\bigg( \int \d x^2 \bigg( \int \d \xi_3 \left\| \int \d \xi_1 U^{(e)}( x^2,\xi_1) f_{1,t}(\xi_1)\overline{G(\xi_1,\xi_3)}\right\|_{\C^4}^{2} \bigg) \bigg)^{\frac{1}{2}}\times\\
  &\left\|F(.,.)\right\|_{L^2(\Gamma_1\times\R^3)} \|(N_{\mu_-} + 1)^{\frac{1}{2}} e^{-itH}\Psi\|.
   \end{split}
\end{equation}

and

\begin{equation}\label{4.15}
  \begin{split}
  &\left\|e^{itH}[ H_I^{(2)},b_{1,+}(f_{1,t})] e^{-itH}\Psi\right\|\leq   \\
  &C\bigg( \int \d x^2 \bigg( \int \d \xi_3 \left\| \int \d \xi_1 U^{(e)}( x^2,\xi_1) f_{1,t}(\xi_1)\overline{G(\xi_1,\xi_3)}\right\|_{\C^4}^{2} \bigg) \bigg)^{\frac{1}{2}}\times\\
  &\left\|F(.,.)\right\|_{L^2(\Gamma_1\times\R^3)} \|(N_{\mu_+} + 1)^{\frac{1}{2}} e^{-itH}\Psi\|.
  \end{split}
\end{equation}

By \eqref{2.18}and \eqref{2.19} we have
\begin{equation}\label{4.16}
  \|H_I\Psi\|\leq a\|H_0\phi\| +b\|\Psi\|
  \end{equation}
with $$a=4\frac{C}{M}\|F(.,.)\|_{L^2(\Gamma_1\times\R^3)}\|G(.,.)\|_{L^2(\Gamma_1\times\R^3)}$$ and $$b=2C\|F(.,.)\|_{L^2(\Gamma_1\times\R^3)}\|G(.,.)\|_{L^2(\Gamma_1\times\R^3)}$$.

Hence we obtain
\begin{equation}\label{4.17}
  \|H_0\Psi\|\leq \tilde{a}\|H\Psi\| +\tilde{b}\|\Psi\|
\end{equation}
with
\begin{equation*}
  \tilde{a}=\frac{1}{1-g_0a}\, \mbox{and} \; \tilde{b}=\frac{g_0b}{1-g_0a}
\end{equation*}
Therefore we have
\begin{equation}\label{4.18}
\begin{split}
  \|(N_{e}  + 1)^{\frac{1}{2}} e^{-itH}\Psi\| &\leq \frac{1}{m_e}( \tilde{a}\|H\Psi\| +(\tilde{b}+ m_e)\|\Psi\|) \\
  \|(N_{\mu_{\pm}}  + 1)^{\frac{1}{2}} e^{-itH}\Psi\| &\leq \frac{1}{m_\mu}( \tilde{a}\|H\Psi\| +(\tilde{b}+ m_\mu)\|\Psi\|).
\end{split}
\end{equation}
where $m_{\mu}$ is the mass of the muon.

Hence we get
\begin{equation}\label{4.19}
 \begin{split}
  &\left\|e^{itH}[ H_I,b_{1,+}(f_{1,t})] e^{-itH}\Psi\right\|\leq   \\
  &2C\bigg( \int \d x^2 \bigg( \int \d \xi_3 \left\| \int \d \xi_1 U^{(e)}( x^2,\xi_1) f_{1,t}(\xi_1)\overline{G(\xi_1,\xi_3)}\right\|_{\C^4}^{2} \bigg) \bigg)^{\frac{1}{2}}\times\\
   &\left\|F(.,.)\right\|_{L^2(\Gamma_1\times\R^3)}\frac{1}{m_\mu}( \tilde{a}\|H\Psi\| +(\tilde{b}+ m_\mu)\|\Psi\|.
   \end{split}
\end{equation}

Moreover we have
\begin{equation}\label{4.20}
  \begin{split}
      & \int \d \xi_3 \left\| \int \d \xi_1 U^{(e)}( x^2,\xi_1) f_{1,t}(\xi_1)\overline{G(\xi_1,\xi_3)}\right\|_{\C^4}^{2}  \\
      &=\sum_{j=1}^4  \int \d \xi_3 \left| \int \d \xi_1 U_{j}^{(e)}( x^2,\xi_1)e^{-itE_n^(e)(p_e^3)} f_{1}(\xi_1)\overline{G(\xi_1,\xi_3)}\right|^{2}
  \end{split}
\end{equation}
where $\big(\bigcup_{j=1}^4 U_{j}^{(e)}( x^2,\xi_1)\big)$ are the four components of the vectors \eqref{A.8} and \eqref{A.11} $\in \C^4$.

Note that
\begin{equation}\label{4.21}
  e^{-itE_n^{(e)}(p_e^3)}=\frac{1}{it}\frac{E_n^{(e)}(p_e^3)}{p_e^3}\frac{d}{dp_e^3}e^{-itE_n^{(e)}(p_e^3)} .
\end{equation}

By \eqref{4.20} and\eqref{4.21}, by a two-fold partial integration with respect to $p_e^3$ and by Hypothesis 4.1 one can show that there exit for every $j$ a function, denoted by $H_j^{(e)}(\xi_1,\xi_3)$, such that
\begin{equation}\label{4.22}
  \begin{split}
  &\sum_{j=1}^4 \int \d x^2 \bigg( \int \d \xi_3 \left| \int \d \xi_1 U_{j}^{(e)}( x^2,\xi_1)e^{-itE_n^{(e)}(p_e^3)} f_{1}(\xi_1)\overline{G(\xi_1,\xi_3)}\right|^{2}\bigg)  \\
  &= \sum_{j=1}^4\frac{1}{t^4} \int \d x^2 \bigg( \int \d \xi_3 \left| \int \d \xi_1 U_{j}^{(e)}( x^2,\xi_1)H_j^{(e)}(\xi_1,\xi_3)e^{-itE_n^{(e)}(p_e^3}\right|^{2}\bigg)\\
  &\leq C_{f_1} \frac{1}{t^4}\sum_{j=1}^4\big(\int \d \xi_1 \d \xi_3 \chi_{f_1}(\xi_1)|H_j^{(e)}(\xi_1,\xi_3)|^2 \big)<\infty.
  \end{split}
\end{equation}
Here $\chi_{f_1}(.)$ is the characteristic function of the support of $f_1(.)$ and \eqref{A.13} is used.

By \eqref{4.6} and \eqref{4.19}-
 \eqref{4.22} the strong limits of $ b_{1,+,t}(f_1)$ on $\mathfrak{F}$ when t goes to $\pm\infty$ and for all $f_1 \in L^2(\Gamma_1)$ exist for every $g \leq g_0$.

By \eqref{4.11}-\eqref{4.13} and by mimicking the proof of \eqref{4.14} and \eqref{4.15} we get

\begin{equation}\label{4.23}
  \begin{split}
  &\mbox{Sup}\bigg(\left\|e^{itH}[ (H_I^{(1)})^*,b_{1,+}^*(f_{1,t})] e^{-itH}\Psi\right\|, \left\|e^{itH}[ (H_I^{(2)})^*,b_{1,+}^*(f_{1,t})] e^{-itH}\Psi\right\|\bigg)   \\
  &\leq C\bigg( \int \d x^2 \bigg( \int \d \xi_3 \left\| \int \d \xi_1 U^{(e)}( x^2,\xi_1) f_{1,t}(\xi_1)\overline{G(\xi_1,\xi_3)}\right\|_{\C^4}^{2} \bigg) \bigg)^{\frac{1}{2}}\\
  &\times \left\|F(.,.)\right\|_{L^2(\Gamma_1\times\R^3)}\frac{1}{m_\mu}( \tilde{a}\|H\Psi\| +(\tilde{b}+ m_\mu)\|\Psi\|).
   \end{split}
\end{equation}

It follows from \eqref{4.10} and \eqref{4.20}-\eqref{4.23} that the strong limits of $ b_{1,+,t}^*(f_1)$  exist when t goes to $\pm\infty$, for all $f_1 \in L^2(\Gamma_1)$ and for every $g \leq g_0$.

We now consider the existence of $ b_{2,\epsilon,\pm\infty}^{\sharp}(f_2)$

Let $\Psi \in D(H)$ and $f_{2,t}(\xi_2)=(e^{-itE^\mu}f_2)(\xi_2)$ with $f_2\in \mathfrak{D}$. By \eqref{4.4}, \eqref{4.5} and the strong differentiability of $e^{itH}$ we get

\begin{equation}\label{4.24}                                                                                                                                    b_{2,+,T}(f_2)\Psi- b_{2,+,T_{0}}(f_2)\Psi=ig \int_{T_{0}}^{T} e^{itH}[ H_I,b_{2,+}(f_{2,t})] e^{-itH}\Psi dt
\end{equation}

with

\begin{equation}\label{4.25}
  \left[H_I^{(1)},b_{2,+}(f_{2,t})\right]\Psi=\left[(H_I^{(2)}),b_{2,+}(f_{2,t})\right]\Psi=\left[(H_I^{(2)})^*,b_{2,+}(f_{2,t})\right]\Psi=0
\end{equation}

and

\begin{equation}\label{4.26}
\begin{split}
&\left[(H_I^{(1)})^*,b_{2,+}(f_{2,t})\right]\Psi=
- \int \d \xi_1 \d \xi_2 \d \xi_3 \d\xi_4 \,\Big(\int \d x^2 \mathrm{e}^{ix^2r^2} \\
&\big( \overline{W}^{(\overline{\nu}_e)}(\xi_3) \gamma^\alpha(1-\gamma_5)U^{(e)}(x^2,\xi_1)\big)
\big(\overline{U}^{(\mu)}(x^2,\xi_2) \gamma_\alpha(1-\gamma_5)U^{(\nu_\mu)}(\xi_4)\big)\Big)\\                                                                   &\overline{F(\xi_{2}, \xi_{4})}\,\overline{G(\xi_{1}, \xi_{3})}\, \overline{f_{2,t}(\xi_2)}
b_{-}(\xi_3)b_{+}(\xi_1)b_{+}(\xi_4)\Psi.
\end{split}
\end{equation}

Similarly we obtain

\begin{equation}\label{4.27}
  \begin{split}
  &\left\|e^{itH}[ H_I,b_{2,+}(f_{2,t})] e^{-itH}\Psi\right\|\leq   \\
  &C\bigg( \int \d x^2 \bigg( \int \d \xi_4 \left\| \int \d \xi_2 U^{(\mu)}( x^2,\xi_2)F(\xi_2,\xi_4) f_{2,t}(\xi_1)\right\|_{\C^4}^{2} \bigg) \bigg)^{\frac{1}{2}}\times\\
  &\left\|G(.,.)\right\|_{L^2(\Gamma_1\times\R^3)} \|(N_{e} + 1)^{\frac{1}{2}} e^{-itH}\Psi\|.
  \end{split}
\end{equation}

It follows from \eqref{4.16}-\eqref{4.18} that
\begin{equation}\label{4.28}
 \|(N_{e}  + 1)^{\frac{1}{2}} e^{-itH}\Psi\|\leq \frac{1}{m_e}( \tilde{a}\|H\Psi\| +(\tilde{b}+ m_e)\|\Psi\|.
\end{equation}

Hence

\begin{equation}\label{4.29}
  \begin{split}
  &\left\|e^{itH}[ H_I,b_{2,+}(f_{2,t})] e^{-itH}\Psi\right\|\leq   \\
  &C\bigg( \int \d x^2 \bigg( \int \d \xi_4 \left\| \int \d \xi_2 U^{(\mu)}( x^2,\xi_2)F(\xi_2,\xi_4) f_{2,t}(\xi_2)\right\|_{\C^4}^{2} \bigg) \bigg)^{\frac{1}{2}}\times\\
  &\left\|G(.,.)\right\|_{L^2(\Gamma_1\times\R^3)}\frac{1}{m_e}( \tilde{a}\|H\Psi\| +(\tilde{b}+ m_e)\|\Psi\| .
  \end{split}
\end{equation}

Moreover we have
\begin{equation}\label{4.30}
  \begin{split}
      & \int \d \xi_4 \left\| \int \d \xi_2 U^{(\mu)}( x^2,\xi_2) f_{2,t}(\xi_2)F(\xi_2,\xi_4)\right\|_{\C^4}^{2}  \\
      &=\sum_{j=1}^4  \int \d \xi_4 \left| \int \d \xi_2 U_{j}^{(\mu)}( x^2,\xi_2)e^{-itE_n^{(\mu)}(p_{\mu}^3)} f_{2,t}(\xi_2)F(\xi_2,\xi_4)\right|^{2}.
  \end{split}
\end{equation}
where $\big(\bigcup_{j=1}^4 U_{j}^{(\mu)}( x^2,\xi_2)\big)$ are the four components of the vectors \eqref{A.8} and \eqref{A.11} $\in \C^4$ for $\alpha=\mu$.

By\eqref{4.30}, by a two-fold partial integration with respect to $p_{\mu}^3$ and by Hypothesis 4.1 one can show that there exits for every $j$ a function, denoted by $H_j^{(\mu)}(\xi_2,\xi_4)$ , such that
\begin{equation}\label{4.31}
  \begin{split}
  &\sum_{j=1}^4 \int \d x^2 \bigg( \int \d \xi_4 \left| \int \d \xi_2 U_{j}^{(\mu)}( x^2,\xi_2)e^{-itE_{n}^{(\mu)}(p_{\mu}^3)} f_{2}(\xi_2)F(\xi_2,\xi_4)\right|^{2}\bigg)  \\
  &= \sum_{j=1}^4\frac{1}{t^4} \int \d x^2 \bigg( \int \d \xi_4 \left| \int \d \xi_1 U_{j}^{(\mu)}( x^2,\xi_2)H_j^{(\mu)}(\xi_2,\xi_4)e^{-itE_n^{(\mu)}(p_{\mu}^3)} \right|^{2}\bigg)\\
  &\leq C_{f_2} \frac{1}{t^4}\sum_{j=1}^4\big(\int \d \xi_2 \d \xi_4 \chi_{f_2}(\xi_2)|H_j^{(\mu)}(\xi_2,\xi_4)|^2 \big)<\infty.
  \end{split}
\end{equation}
Here $\chi_{f_2}(.)$ is the characteristic function of the support of $f_2(.)$ and \eqref{A.13} is used.

Similarly we have

\begin{equation}\label{4.32}
 b_{2,+,T}^*(f_2)\Psi- b_{2,+,T_{0}}^*(f_2)\Psi=ig \int_{T_{0}}^{T} e^{itH}[ H_I,b_{2,+}^*(f_{2,t})] e^{-itH}\Psi dt
\end{equation}

with

\begin{equation}\label{4.33}
  \left[(H_I^{(1)})^*,b_{2,+}^*(f_{2,t})\right]\Psi=\left[H_I^{(2)},b_{2,+}^*(f_{2,t})\right]\Psi=\left[(H_I^{(2)})^*,b_{2,+}^*(f_{2,t})\right]\Psi=0
\end{equation}

and

\begin{equation}\label{4.34}
\begin{split}
&\left[(H_I^{(1)}),b_{2,+}^*(f_{2,t})\right]\Psi=
- \int \d \xi_1 \d \xi_2 \d \xi_3 \d\xi_4 \,\Big(\int \d x^2 \mathrm{e}^{-ix^2r^2} \\
&\big(\overline{U}^{(e)}(x^2,\xi_1)\gamma^\alpha(1-\gamma_5) W^{(\overline{\nu}_e)}(\xi_3)\big)G(\xi_{1}, \xi_{3})
\big(\overline{U}^{(\nu_\mu)}(\xi_4)\gamma_\alpha(1-\gamma_5)U^{(\mu)}(x^2,\xi_2)\big)\Big)\\                                                                                   &F(\xi_{2}, \xi_{4})G(\xi_{1}, \xi_{3})f_{2,t}(\xi_2)
b_{+}^*(\xi_4)b_{+}^*(\xi_1)b_{-}^*(\xi_3)\Psi.
\end{split}
\end{equation}

Similarly we obtain

\begin{equation}\label{4.35}
  \begin{split}
  &\left\|e^{itH}[ H_I,b_{2,+}^*(f_{2,t})] e^{-itH}\Psi\right\|\leq   \\
  &C\bigg( \int \d x^2 \bigg( \int \d \xi_4 \left\| \int \d \xi_2 U^{(\mu)}( x^2,\xi_2)F(\xi_2,\xi_4) f_{2,t}(\xi_1)\right\|_{\C^4}^{2} \bigg) \bigg)^{\frac{1}{2}}\times\\
  &\left\|G(.,.)\right\|_{L^2(\Gamma_1\times\R^3)}\frac{1}{m_e}( \tilde{a}\|H\Psi\| +(\tilde{b}+ m_e)\|\Psi\| .
  \end{split}
\end{equation}

It follows from \eqref{4.29},\eqref{4.31} and \eqref{4.35} that the strong limits of $b^{\sharp}(2,+,t)(f_2)$ exist when t goes to $\pm\infty$, for all $f_2\in L^2(\Gamma_1\times\R^3)$ and for every $g \leq g_0$.

Let us now consider the strong limits of $b^{\sharp}(2,-,t)(f_2)$.

We have for all $f_2\in \mathfrak{D}$
\begin{equation}\label{4.36}
b_{2,-,T}(f_)\Psi- b_{2,-,T_{0}}(f_2)\Psi=ig \int_{T_{0}}^{T} e^{itH}[ H_I,b_{2,-}(f_{2,t})] e^{-itH}\Psi dt
\end{equation}

with

\begin{equation}\label{4.37}
  \left[H_I^{(1)},b_{2,-}(f_{2,t})\right]\Psi=\left[(H_I^{(1)})^*,b_{2,-}(f_{2,t})\right]\Psi=\left[(H_I^{(2)})^*,b_{2,-}(f_{2,t})\right]\Psi=0
\end{equation}

\begin{equation}\label{4.38}
\begin{split}
&\left[(H_I^{(2)}),b_{2,-}(f_{2,t})\right]\Psi=
- \int \d \xi_1 \d \xi_2 \d \xi_3 \d\xi_4 \,\Big(\int \d x^2 \mathrm{e}^{-ix^2r^2} \\
&\big(\overline{U}^{(e)}(x^2,\xi_1)\gamma^\alpha(1-\gamma_5) W^{(\overline{\nu}_e)}(\xi_3)\big)G(\xi_{1}, \xi_{3})
\big(\overline{U}^{(\nu_\mu)}(\xi_4)\gamma_\alpha(1-\gamma_5)W^{(\mu)}(x^2,\xi_2)\big)\Big)\\                                                                                   &F(\xi_{2}, \xi_{4})G(\xi_{1}, \xi_{3})f_{2,t}(\xi_2)
b_{+}^*(\xi_4)b_{+}^*(\xi_1)b_{-}^*(\xi_3)\Psi.
\end{split}
\end{equation}

By mimicking the proofs given above we get

\begin{equation}\label{4.39}
  \begin{split}
  &\left\|e^{itH}[ H_I,b_{2,-}(f_{2,t})] e^{-itH}\Psi\right\|\leq   \\
  &C\bigg( \int \d x^2 \bigg( \int \d \xi_4 \left\| \int \d \xi_2 W^{(\mu)}( x^2,\xi_2)F(\xi_2,\xi_4)\overline{f_{2,t}(\xi_1)}\right\|_{\C^4}^{2} \bigg) \bigg)^{\frac{1}{2}}\times\\
  &\left\|G(.,.)\right\|_{L^2(\Gamma_1\times\R^3)}\frac{1}{m_e}( \tilde{a}\|H\Psi\| +(\tilde{b}+ m_e)\|\Psi\| .
  \end{split}
\end{equation}

and

\begin{equation}\label{4.40}
  \begin{split}
      & \int \d \xi_4 \left\| \int \d \xi_2 W^{(\mu)}( x^2,\xi_2) \overline{f_{2,t}(\xi_2)}F(\xi_2,\xi_4)\right\|_{\C^4}^{2}  \\
      &=\sum_{j=1}^4 \int \d \xi_4 \left| \int \d \xi_2 W_{j}^{(\mu)}( x^2,\xi_2)e^{-itE_n^{(\mu)}(p_{\mu}^3)} \overline{f_{2}(\xi_2)}F(\xi_2,\xi_4) \right|^{2}.
  \end{split}
\end{equation}
where $\big(\bigcup_{j=1}^4 W_{j}^{(\mu)}( x^2,\xi_2)\big)$ are the four components of the vectors \eqref{A.14} - \eqref{A.16} $\in \C^4$ for $\alpha=\mu$.

By\eqref{4.40}, by a two-fold partial integration with respect to $p_{\mu}^3$ and by Hypothesis 4.1 one can show that there exit for every $j$ a function, denoted by $\widetilde{H}_j^{(\mu)}(\xi_2,\xi_4)$ , such that
\begin{equation}\label{4.41}
  \begin{split}
  &\sum_{j=1}^4 \int \d x^2 \bigg( \int \d \xi_4 \left| \int \d \xi_2 W_{j}^{(\mu)}( x^2,\xi_2)e^{-itE_{n}^{(\mu)}(p_{\mu}^3)} \overline{f_{2}(\xi_2)}F(\xi_2,\xi_4)\right|^{2}\bigg)  \\
  &= \sum_{j=1}^4\frac{1}{t^4} \int \d x^2 \bigg( \int \d \xi_4 \left| \int \d \xi_2 W_{j}^{(\mu)}( x^2,\xi_2)\widetilde{H}_j^{(\mu)}(\xi_2,\xi_4)e^{-itE_n^{(\mu)}(p_{\mu}^3)} \right|^{2}\bigg)\\
  &\leq C_{f_2} \frac{1}{t^4}\sum_{j=1}^4\big(\int \d \xi_2 \d \xi_4 \chi_{f_2}(\xi_2)|\widetilde{H}_j^{(\mu)}(\xi_2,\xi_4)|^2 \big)<\infty.
  \end{split}
\end{equation}
Here $\chi_{f_2}(.)$ is the characteristic function of the support of $f_2(.)$ and \eqref{A.17} is used.

It follows from \eqref{4.36},\eqref{4.39}-\eqref{4.41} that the strong limits of $b_{2,-,t}(f_2)$ exist when t goes to $\pm\infty$, for all $f_2\in L^2(\Gamma_1\times\R^3)$ and for every $g \leq g_0$.

We now have for all $f_2\in \mathfrak{D}$
\begin{equation}\label{4.42}
 b_{2,-,T}^{*}(f_2)\Psi- b_{2,-,T_{0}}^*(f_2)\Psi=ig \int_{T_{0}}^{T} e^{itH}[ H_I,b_{2,-}^*(f_{2,t})] e^{-itH}\Psi dt
\end{equation}
 with

\begin{equation}\label{4.43}
  \left[H_I^{(1)},b_{2,-}^*(f_{2,t})\right]\Psi=\left[(H_I^{(1)})^*,b_{2,-}^*(f_{2,t})\right]\Psi=\left[H_I^{(2)},b_{2,-}^*(f_{2,t})\right]\Psi=0
\end{equation}

\begin{equation}\label{4.44}
\begin{split}
&\left[(H_I^{(2)})^*,b_{2,-}^*(f_{2,t})\right]\Psi=
- \int \d \xi_1 \d \xi_2 \d \xi_3 \d\xi_4 \,\Big(\int \d x^2 \mathrm{e}^{ix^2r^2} \\
&\big( \overline{W}^{(\overline{\nu}_e)}(\xi_3)\gamma^\alpha(1-\gamma_5)U^{(e)}(x^2,\xi_1)\big)
\big(\overline{W}^{(\mu)}(x^2,\xi_2)\gamma_\alpha(1-\gamma_5)U^{(\nu_\mu)}(\xi_4)\big)\Big)\\                                                                         &\overline{F(\xi_{2}, \xi_{4})}\,\overline{G(\xi_{1}, \xi_{3})}f_{2,t}(\xi_2)
b_{-}(\xi_3)b_{+}(\xi_1)b_{+}(\xi_4)\Psi.
\end{split}
\end{equation}

Similarly to \eqref{4.39} we get

\begin{equation}\label{4.45}
  \begin{split}
  &\left\|e^{itH}[ H_I,b_{2,-}^*(f_{2,t})] e^{-itH}\Psi\right\|\leq   \\
  &C\bigg( \int \d x^2 \bigg( \int \d \xi_4 \left\| \int \d \xi_2 W^{(\mu)}( x^2,\xi_2)F(\xi_2,\xi_4)\overline{f_{2,t}(\xi_1)}\right\|_{\C^4}^{2} \bigg) \bigg)^{\frac{1}{2}}\times\\
  &\left\|G(.,.)\right\|_{L^2(\Gamma_1\times\R^3)}\frac{1}{m_e}( \tilde{a}\|H\Psi\| +(\tilde{b}+ m_e)\|\Psi\| .
  \end{split}
\end{equation}

It follows from \eqref{4.43},\eqref{4.45},\eqref{4.40} and \eqref{4.41} that the strong limits of $b(2,-,t)^*(f_2)$ exist when t goes to $\pm\infty$, for all $f_2\in L^2(\Gamma_1\times\R^3)$ and for every $g \leq g_0$.

\noindent \textbf{Strong limits of $b_{3,-,t}^{\sharp}(f_3)$ and $b_{4,+,t}^{\sharp}(f_4)$ }.

Let
\begin{equation}\label{4.46}
\mathfrak{D'}= \{f(.)\in C_0^\infty(\R^3 \setminus \{(0,0,p^3)\}); p^3\in\R \}.
\end{equation}

 Let $f_1,f_2\in \mathfrak{D'}$ . According to \cite[lemma1]{HK1} we have
 \begin{equation}\label{4.47}
   b_{4,+}^{\sharp}(f_4)D(H) \subset D(H)\quad \mbox{and}\quad b_{3,-}^{\sharp}(f_3)D(H)\subset D(H).
 \end{equation}

 Moreover we have
 \begin{equation}\label{4.48}
 \begin{split}
 e^{itH_0}b_{3,-}^{\sharp}(f_3)e^{-itH_0}\Psi=&b_{3,-}^{\sharp}(e^{it|\textbf{p}_3|}f_3)\Psi, \\
 e^{itH_0}b_{4,+}^{\sharp}(f_4)e^{-itH_0}\Psi=&b_{4,+}^{\sharp}(e^{it|\textbf{p}_4|}f_4)\Psi.
 \end{split}
 \end{equation}
where $\Psi \in D(H)$.

Let $\Psi \in D(H)$ and $f_{j,t}(\xi_j)=(e^{-it|\textbf{p}_j|}f_j)(\xi_j)$ where $j=3,4$.. By \eqref{4.4},\eqref{4.5} and the strong differentiability of $e^{itH}$ we get

\begin{equation}\label{4.49}
 b_{3,-,T}(f_3)\Psi- b_{3,-,T_{0}}(f_1)\Psi=ig \int_{T_{0}}^{T} e^{itH}[ H_I,b_{3,-}(f_{3,t})] e^{-itH}\Psi dt
\end{equation}

By using the usual anticommutation relations (CAR)(see \eqref{A.4} and \eqref{B.4}) we easily get for all $\Psi \in D(H)$
\begin{equation}\label{4.50}
\begin{split}
&\left[H_I^{(1)},b_{3,-}(f_{3,t})\right]\Psi=
-\int \d \xi_1 \d \xi_2 \d \xi_3 \d \xi_4 \,\Big(\int \d x^2 \mathrm{e}^{-ix^2r^2} \\
&\big( \overline{U}^{(\nu_\mu)}(\xi_4) \gamma_{\alpha}(1-\gamma_5)U^{(\mu)}(x^2,\xi_2)\big)
\big(\overline{U}^{(e)}(x^2,\xi_1) \gamma^\alpha(1-\gamma_5)W^{(\overline{\nu}_{e})}(\xi_3)\big)\Big)\\                                         &\overline{f_{3,t}(\xi_3)}F(\xi_{2}, \xi_{4})G(\xi_1,\xi_3)
b_{+}^{*}(\xi_4)b_{+}^{*}(\xi_1)b_{+}(\xi_2)\Psi.
\end{split}
\end{equation}

\begin{equation}\label{4.51}
\begin{split}
&\left[H_I^{(2)},b_{3,-}(f_{3,t})\right]\Psi=
 \int \d \xi_1 \d \xi_2 \d \xi_3 \d \xi_4 \Bigg(\int \d x^2  \mathrm{e}^{-ix^2r^2} \\
&\big( \overline{U}^{(\nu_\mu)}(\xi_4) \gamma^\alpha(1-\gamma_5)W^{(\mu)}(x^2,\xi_2)\big)
\big(\overline{U}^{(e)}(x^2,\xi_1)\gamma_{\alpha}(1-\gamma_5)W^{(\overline{\nu}_{e})}(\xi_3)\big)\Big)\\                  &\overline{f_{3,t}(\xi_3)}F(\xi_{2},\xi_{4})G(\xi_1,\xi_3)
b_{+}^{*}(\xi_4)b^{*}_{-}(\xi_2)b^*_{+}(\xi_1)\Psi.
\end{split}
\end{equation}

and

\begin{equation}\label{4.52}
  \left[(H_I^{(1)})^*,b_{3,-}(f_{3,t})\right]\Psi=\left[(H_I^{(2)})^*,b_{3,-}(f_{3,t})\right]\Psi=0
\end{equation}

By \eqref{B.5} we get
\begin{equation}\label{4.53}
\begin{split}
&\left\|\left[H_I^{(1)},b_{3,-}(f_{3,t})\right]\Psi\right\|\leq \\
&\int \d x^2 \bigg(\int \d \xi_1\left|\big(\overline{U}^{(e)}(x^2,\xi_1) \gamma^\alpha(1-\gamma_5) \int \d\xi_3 \mathrm{e}^{-ip^2_{3}x^2} W^{(\overline{\nu}_{e})}(\xi_3)G(\xi_1,\xi_3)\overline{f_{3,t}(\xi_3)}\big)\right|^2\bigg)^\frac{1}{2}\\                                                                                          &\left \|\int \d \xi_2 \d \xi_4 \mathrm{e}^{-ip^2_{4}x^2}\big( \overline{U}^{(\nu_\mu)}(\xi_4) \gamma_{\alpha}(1-\gamma_5)U^{(\mu)}(x^2,\xi_2)\big)b_{+}^{*}(\xi_4)b_{+}(\xi_2)\Psi\right\|.
\end{split}
\end{equation}

and

\begin{equation}\label{4.54}
\begin{split}
&\left\|\left[H_I^{(2)},b_{3,-}(f_{3,t})\right]\Psi\right\|\leq \\
&\int \d x^2 \bigg(\int \d \xi_1\left|\big(\overline{U}^{(e)}(x^2,\xi_1) \gamma^\alpha(1-\gamma_5) \int \d\xi_3 \mathrm{e}^{-ip^2_{3}x^2} W^{(\overline{\nu}_{e})}(\xi_3)G(\xi_1,\xi_3)\overline{f_{3,t}(\xi_3)}\big)\right|^2\bigg)^\frac{1}{2}\\                                                                                          &\left \|\int \d \xi_2 \d \xi_4 \mathrm{e}^{-ip^2_{4}x^2}\big( \overline{W}^{(\nu_\mu)}(\xi_4) \gamma_{\alpha}(1-\gamma_5)U^{(\mu)}(x^2,\xi_2)\big)b_{+}^{*}(\xi_4)b_{-}^*(\xi_2)\Psi\right\|.
\end{split}
\end{equation}

Moreover we have
\begin{equation}\label{4.55}
  \begin{split}
      & \int \d \xi_1 \left\| \int \d \xi_3 \mathrm{e}^{-ip^2_{3}x^2} W^{(\overline{\nu_e})}(\xi_3) \overline{f_{1,t}(\xi_1)}G(\xi_1,\xi_3)\right\|_{\C^4}^{2}  \\
      &=\sum_{j=1}^4  \int \d \xi_1 \left| \int \d \xi_3 \mathrm{e}^{-ip^2_{3}x^2}W_j^{(\overline{\nu_e})}(\xi_3)(e^{it|\textbf{p}_3|}\overline{f_{3}(\xi_3)}G(\xi_1,\xi_3)\right|^{2}.
  \end{split}
\end{equation}
where $\big(\bigcup_{j=1}^4 W_{j}^{(\overline{\nu_e})}(\xi_3)\big)$ are the four components of the vector \eqref{B.12} $\in \C^4$.

By a two-fold partial integration with respect to $p^3$ and $p^1$ and by Hypothesis 4.2 one can show that there exit for every $j$ a function , denoted by $H_j^{(\overline{\nu_e})}(\xi_1,\xi_3)$, such that
\begin{equation}\label{4.56}
  \begin{split}
  &\sum_{j=1}^4 \int \d \xi_1 \left| \int \d \xi_3\mathrm{e}^{-ip^2_{3}x^2} W_{j}^{(\overline{\nu_e})}(\xi_3)e^{it|\textbf{p}_3|} \overline{f_{3}(\xi_3)}G(\xi_1,\xi_3)\right|^{2}\\
  &= \sum_{j=1}^4\frac{1}{t^4}  \int \d \xi_1 \left| \int \d \xi_3 \mathrm{e}^{-ip^2_{3}x^2} W_{j}^{(\overline{\nu_e})}(\xi_3)H_j^{(\overline{\nu_e})}(\xi_1,\xi_3)e^{it|\textbf{p}_3|}\right|^{2}\\
  &\leq C_{f_3}^2 \frac{1}{t^4}\sum_{j=1}^4\int \d \xi_1 \d \xi_3 \chi_{f_3}(\xi_3)e^{it|\textbf{p}_3|}|H_j^{(\overline{\nu_e})}(\xi_1,\xi_3)|^2<\infty.
  \end{split}
\end{equation}
Here $\chi_{f_3}(.)$ is the characteristic function of the support of $f_3(.)$.

By the $N_\tau$ estimates and by \eqref{4.18},\eqref{A.13}, \eqref{A.17} and \eqref{B.14} it follows from \eqref{4.52}-\eqref{4.56} that, for every $\Psi \in D(H)$,
\begin{equation}\label{4.57}
  \begin{split}
  &\left\|e^{itH}\left [H_I,b_{3,-}(f_{3,t})\right]e^{-itH}\Psi\right\|\leq  \\
  &CC_{f_3}\frac{1}{t^2}\bigg(\sum_{j=1}^4\int \d \xi_1 \d \xi_3 \chi_{f_3}(\xi_3)|H_j^{(\overline{\nu_e})}(\xi_1,\xi_3)|^2\bigg)^{\frac{1}{2}}\times \\
  &\left\|F(.,.)\right\|_{L^2(\Gamma_1\times\R^3)}\frac{1}{m_\mu}( \tilde{a}\|H\Psi\| +(\tilde{b}+ m_\mu)\|\Psi\|)
  \end{split}
\end{equation}

Furthermore we have

\begin{equation}\label{4.58}
b^*_{3,-,T}(f_3)\Psi- b^*_{3,-,T_{0}}(f_3)\Psi=ig \int_{T_{0}}^{T} e^{itH} [ H_I,b^*_{3,-}(f_{3,t})] e^{-itH}\Psi dt
\end{equation}

with

\begin{equation}\label{4.59}
\begin{split}
&\left[(H_I^{(1)})^*,b_{3,-}^*(f_{3,t})\right]\Psi=
- \int \d \xi_1 \d \xi_2 \d \xi_3 \d\xi_4 \,\Big(\int \d x^2 \mathrm{e}^{ix^2r^2} \\
&\big( \overline{W}^{(\overline{\nu}_e)}(\xi_3)\gamma^\alpha(1-\gamma_5)U^{(e)}(x^2,\xi_1)\big)
\big(\overline{U}^{(\mu)}(x^2,\xi_2)\gamma_\alpha(1-\gamma_5)U^{(\nu_\mu)}(\xi_4)\big)\Big)\\                                                                         &\overline{F(\xi_{2}, \xi_{4})}\,\overline{G(\xi_{1}, \xi_{3})}f_{3,t}(\xi_2)
b_{+}^*(\xi_2)b_{+}(\xi_1)b_{+}(\xi_4)\Psi.
\end{split}
\end{equation}

\begin{equation}\label{4.60}
  \begin{split}
  &\left[(H_I^{(2)})^{*},b^*_{3,-}(f_{3,t})\right]\Psi=
   \int \d \xi_1 \d \xi_2 \d \xi_3 \d \xi_4\,\Big(\int \d x^2  \mathrm{e}^{ix^2r^2} \\
  &\big(\overline{W}^{(\overline{\nu}_e)}(\xi_3) \gamma^\alpha(1-\gamma_5)U^{(e)}(x^2,\xi_1)\big)
  \big( W^{(\mu)}(x^2,\xi_2) \gamma_\alpha(1-\gamma_5) U^{(\nu_\mu)}(\xi_4)\big)\Big)\\                                                                              &\overline{F((\xi_{2}, \xi_{4})}\,\overline{G(\xi_1,\xi_3)}f_{3,t}
  b_{+}(\xi_1)b_{-}(\xi_2)b_{+}(\xi_4)\Psi.
  \end{split}
\end{equation}

and

\begin{equation}\label{4.61}
  \left[(H_I^{(1)}),b_{3,-}^*(f_{3,t})\right]\Psi=\left[(H_I^{(2)}),b_{3,-}^*(f_{3,t})\right]\Psi=0
\end{equation}

By adapting the proof of\eqref{4.53}-\eqref{4.57} to \eqref{4.58}-\eqref{4.61} we obtain
\begin{equation}\label{4.62}
  \begin{split}
  &\left\|e^{itH}[ H_I,b_{3,-}^*(f_{3,t})] e^{-itH}\Psi\right\|\\
  &\leq CC_{f_3}\frac{1}{t^2}\bigg(\sum_{j=1}^4\int \d \xi_1 \d \xi_3 \chi_{f_3}(\xi_3)|H_j^{(\overline{\nu_e})}(\xi_1,\xi_3)|^2\bigg)^{\frac{1}{2}}\times \\
  &\times \left\|F(.,.)\right\|_{L^2(\Gamma_1\times\R^3)}\frac{1}{m_\mu}( \tilde{a}\|H\Psi\| +(\tilde{b}+ m_\mu)\|\Psi\|).
   \end{split}
\end{equation}
Here $\chi_{f_3}(.)$ is the characteristic function of the support of $f_3(.)$.

It follows from \eqref{4.49},\eqref{4.47},\eqref{4.58} and \eqref{4.62} that the strong limits of $b^\sharp_{3,-,t}(f 3 )$ exist when t goes to $\pm\infty$, for all $f_3\in L^2(\R^3)$ and for every $g \leq g_0$.

We now have

\begin{equation}\label{4.63}
b_{4,+,T}(f_1)\Psi- b_{4,+,T_{0}}(f_1)\Psi=ig \int_{T_{0}}^{T} e^{itH}[ H_I,b_{4,+}(f_{4,t})] e^{-itH}\Psi dt
\end{equation}

By using the usual canonical anticommutation relations (CAR)(see \eqref{A.4} and \eqref{B.4}) we easily get for all $\Psi \in D(H)$

\begin{equation}\label{4.64}
\begin{split}
&\left[H_I^{(1)},b_{4,+}(f_{4,t})\right]\Psi=
-\int \d \xi_1 \d \xi_2 \d \xi_3 \d \xi_4 \Big(\int \d x^2 \mathrm{e}^{-ix^2r^2} \\
&\big( \overline{U}^{(\nu_\mu)}(\xi_4) \gamma_{\alpha}(1-\gamma_5)U^{(\mu)}(x^2,\xi_2)\big)
\big(\overline{U}^{(e)}(x^2,\xi_1) \gamma^\alpha(1-\gamma_5)W^{(\overline{\nu}_{e})}(\xi_3)\big)\Big)\\                                                &\overline{f_{4,t}(\xi_4)}F(\xi_{2}, \xi_{4})G(\xi_1,\xi_3)
b_{+}^{*}(\xi_1)b_{-}^{*}(\xi_3)b_{+}(\xi_2)\Psi.
\end{split}
\end{equation}

\begin{equation}\label{4.65}
\begin{split}
&\left[H_I^{(2)},b_{4,+}(f_{4,t})\right]\Psi=
-\int \d \xi_1 \d \xi_2 \d \xi_3 \d \xi_4 \Bigg(\int \d x^2  \mathrm{e}^{-ix^2r^2} \\
&\big( \overline{U}^{(\nu_\mu)}(\xi_4) \gamma^\alpha(1-\gamma_5)W^{(\mu)}(x^2,\xi_2)\big)
\big(\overline{U}^{(e)}(x^2,\xi_1)\gamma_{\alpha}(1-\gamma_5)W^{(\overline{\nu}_{e})}(\xi_3)\big)\Big)\\                         &\overline{f_{4,t}(\xi_4)}F(\xi_{2},\xi_{4})G(\xi_1,\xi_3)
b_{-}^{*}(\xi_2)b^{*}_{+}(\xi_1)b^*_{-}(\xi_3)\Psi.
\end{split}
\end{equation}

and

\begin{equation}\label{4.66}
  \left[(H_I^{(1)})^*,b_{4,+}(f_{4,t})\right]\Psi=\left[(H_I^{(2)})^*,b_{4,+}(f_{4,t})\right]\Psi=0
\end{equation}

By \eqref{B.5} we get
\begin{equation}\label{4.67}
\begin{split}
&\left\|\left[H_I^{(1)},b_{4,+}(f_{4,t})\right]\Psi\right\|\leq \\
&\int \d x^2 \bigg(\int \d \xi_2 \left| \left\langle \int \d \xi_4 U^{(\nu_\mu)}(\xi_4)f_{4,t}(\xi_4)\mathrm{e}^{ip_{4}^2 x^2} \overline{F(\xi_{2},\xi_{4})}, \gamma^{0}\gamma^\alpha(1-\gamma_5) U^{(\mu)}(x^2,\xi_2) \right\rangle \right|^2 \bigg)^{\frac{1}{2}}\\                                                                                         &\left \|\int \d \xi_1 \d \xi_3 \mathrm{e}^{-ip^2_{3}x^2}\left\langle U^{(e)}(x^2,\xi_1),\gamma^0\gamma_{\alpha}(1-\gamma_5)W^{(\overline{\nu_e})}(\xi_3)G(\xi_1,\xi_3)\right\rangle b_{+}^{*}(\xi_1)b_{-}^*(\xi_3)\Psi\right\|.
\end{split}
\end{equation}
where $\left\langle .,.\right\rangle$ is the scalar product in $\C^4$.

and

\begin{equation}\label{4.68}
\begin{split}
&\left\|\left[H_I^{(2)},b_{4,+}(f_{4,t})\right]\Psi\right\|\leq \\
&\int \d x^2 \bigg(\int \d \xi_2 \left| \left\langle \int \d \xi_4 U^{(\nu_\mu)}(\xi_4)f_{4,t}(\xi_4)\mathrm{e}^{ip_{4}^2 x^2} \overline{F(\xi_{2},\xi_{4})}, \gamma^0 \gamma^\alpha(1-\gamma_5) W^{(\mu)}(x^2,\xi_2) \right\rangle \right|^2 \bigg)^{\frac{1}{2}}\\                                                                                         &\left \|\int \d \xi_1 \d \xi_3 \mathrm{e}^{-ip^2_{3}x^2}\left\langle U^{(e)}(x^2,\xi_1),\gamma^0
\gamma_{\alpha}(1-\gamma_5)W^{(\overline{\nu_e})}(\xi_3)G(\xi_1,\xi_3)\right\rangle b_{+}^{*}(\xi_1)b_{-}^*(\xi_3)\Psi\right\|.
\end{split}
\end{equation}

By adapting the proof of \eqref{4.57} to \eqref{4.67} and \eqref{4.68} one can show that there exist for every $j$ a function, denoted by $H^{\nu_\mu}(\xi_2,\xi_4)$, such that

\begin{equation}\label{4.69}
  \begin{split}
  &\left\|e^{itH}\left[H_I,b_{4,+}(f_{4,t})\right]e^{-itH}\Psi\right\|\leq  \\
  &CC_{f_4}\frac{1}{t^2}\bigg(\sum_{j=1}^4\int \d \xi_2 \d \xi_4 \chi_{f_4}(\xi_4)|H_j^{(\nu_\mu)}(\xi_2,\xi_4)|^2\bigg)^{\frac{1}{2}}\times \\
  &\left\|G(.,.)\right\|_{L^2(\Gamma_1\times\R^3)}\frac{1}{m_e}( \tilde{a}\|H\Psi\| +(\tilde{b}+ m_e)\|\Psi\|)
  \end{split}
\end{equation}

with
\begin{equation*}
  \sum_{j=1}^4\int \d \xi_2 \d \xi_4 \chi_{f_4}(\xi_4)|H_j^{(\nu_\mu)}(\xi_2,\xi_4)|^2<\infty
\end{equation*}
Here $\chi_{f_4}(.)$ is the characteristic function of the support of $f_4(.)$.

Similarly we have

\begin{equation}\label{4.70}
b^*_{4,+,T}(f_4)\Psi- b^*_{4,+,T_{0}}(f_4)\Psi=ig \int_{T_{0}}^{T} e^{itH} [ H_I,b^*_{4,+}(f_{4,t})] e^{-itH}\Psi dt
\end{equation}

with

\begin{equation}\label{4.71}
\begin{split}
&\left[(H_I^{(1)})^*,b_{4,+}^*(f_{4,t})\right]\Psi=
 \int \d \xi_1 \d \xi_2 \d \xi_3 \d\xi_4 \,\Bigg(\int \d x^2 \mathrm{e}^{ix^2r^2} \\
&\big( \overline{W}^{(\overline{\nu}_e)}(\xi_3)\gamma^\alpha(1-\gamma_5)U^{(e)}(x^2,\xi_1)\big)
\big(\overline{U}^{(\mu)}(x^2,\xi_2)\gamma_\alpha(1-\gamma_5)U^{(\nu_\mu)}(\xi_4)\big)\Big)\\                                                                         &\overline{F(\xi_{2}, \xi_{4})}\,\overline{G(\xi_{1}, \xi_{3})}f_{4,t}(\xi_4)
b_{+}^*(\xi_2)b_{-}(\xi_3)b_{+}(\xi_1)\Psi.
\end{split}
\end{equation}

\begin{equation}\label{4.72}
  \begin{split}
  &\left[(H_I^{(2)})^{*},b^*_{4,+}(f_{4,t})\right]\Psi=
   \int \d \xi_1 \d \xi_2 \d \xi_3 \d \xi_4\,\Bigg(\int \d x^2  \mathrm{e}^{ix^2r^2} \\
  &\big(\overline{W}^{(\overline{\nu}_e)}(\xi_3) \gamma^\alpha(1-\gamma_5)U^{(e)}(x^2,\xi_1)\big)
  \big(\overline{W}^{(\mu)}(x^2,\xi_2) \gamma_\alpha(1-\gamma_5) U^{(\nu_\mu)}(\xi_4)\big)\Big)\\                                                                    &\overline{F((\xi_{2}, \xi_{4})}\,\overline{G(\xi_1,\xi_3)}f_{4,t}
  b_{-}(\xi_3)b_{+}(\xi_1)b_{-}(\xi_2)\Psi.
  \end{split}
\end{equation}

and

\begin{equation}\label{4.73}
  \left[H_I^{(1)},b_{4,+}^*(f_{4,t})\right]=\left[(H_I^{(2)},b^*_{4,+}(f_{4,t})\right]=0
\end{equation}

By \eqref{B.5} we get
\begin{equation}\label{4.74}
\begin{split}
&\left\|\left[(H_I^{(1)})^*,b^*_{4,+}(f_{4,t})\right]\Psi\right\|\leq \\
&\int \d x^2 \bigg(\int \d \xi_2 \left| \left\langle  U^{(\mu)}(x^2,\xi_2), \gamma^{0}\gamma^\alpha(1-\gamma_5)\int \d \xi_4 U^{(\nu_\mu)}(\xi_4)f_{4,t}(\xi_4)\mathrm{e}^{ip_{4}^2x^2} \overline{F(\xi_{2},\xi_{4})}\right\rangle \right|^2 \bigg)^{\frac{1}{2}}\\                                                                                                   &\left \|\int \d \xi_1 \d \xi_3 \mathrm{e}^{-ip^2_{3}x^2}\left\langle W^{(\overline{\nu_e})}(\xi_3),\gamma^0\gamma_{\alpha}(1-\gamma_5)U^{(e)}(x^2,\xi_1)\overline{G(\xi_1,\xi_3)}\right\rangle b_{+}(\xi_1)b_{-}(\xi_3)\Psi\right\|.
\end{split}
\end{equation}
where $\left\langle .,.\right\rangle$ is the scalar product in $\C^4$.

and

\begin{equation}\label{4.75}
\begin{split}
&\left\|\left[(H_I^{(2  )})^*,b^*_{4,+}(f_{4,t})\right]\Psi\right\|\leq \\
&\int \d x^2 \bigg(\int \d \xi_2 \left| \left\langle  U^{(\mu)}(x^2,\xi_2), \gamma^{0}\gamma^\alpha(1-\gamma_5)\int \d \xi_4 U^{(\nu_\mu)}(\xi_4)f_{4,t}(\xi_4)\mathrm{e}^{ip_{4}^2x^2} \overline{F(\xi_{2},\xi_{4})}\right\rangle \right|^2 \bigg)^{\frac{1}{2}}\\                                                                                                   &\left \|\int \d \xi_1 \d \xi_3 \mathrm{e}^{-ip^2_{3}x^2}\left\langle W^{(\overline{\nu_e})}(\xi_3),\gamma^0\gamma_{\alpha}(1-\gamma_5)U^{(e)}(x^2,\xi_1)\overline{G(\xi_1,\xi_3)}\right\rangle b_{+}(\xi_1)b_{-}(\xi_3)\Psi\right\|.
\end{split}
\end{equation}

By adapting the proof of \eqref{4.57} and \eqref{4.67} to \eqref{4.74} and \eqref{4.75} one gets

\begin{equation}\label{4.76}
  \begin{split}
  &\left\|e^{itH}\left[H_I,b^*_{4,+}(f_{4,t})\right]e^{-itH}\Psi\right\|\leq  \\
  &CC_{f_4}\frac{1}{t^2}\bigg(\sum_{j=1}^4\int \d \xi_2 \d \xi_4 \chi_{f_4}(\xi_4)|H_j^{(\nu_\mu)}(\xi_2,\xi_4)|^2\bigg)^{\frac{1}{2}}\times \\
  &\left\|G(.,.)\right\|_{L^2(\Gamma_1\times\R^3)}\frac{1}{m_e}( \tilde{a}\|H\Psi\| +(\tilde{b}+ m_e)\|\Psi\|)
  \end{split}
\end{equation}

It follows from \eqref{4.63},\eqref{4.69},\eqref{4.70} and \eqref{4.76} that the strong limits of $b^\sharp_{4,+,t}(f_4)$ exist  when t goes to $\pm\infty$, for all $f_4\in L^2(\R^3)$ and for every $g \leq g_0$.

This concludes the proof of theorem 4.3.

\end{proof}

\subsection{Existence of a Fock space subrepresentation of the asymptotic CAR}\mbox{}

From now on we only consider the case where the time t goes to $+\infty$. The following proposition is an easy consequence of theorem 4.1.

\begin{proposition}\label{4.5}\mbox{}

Suppose Hypothesis 2.1-Hypothesis 4.2 and $g \leq g_1$. We have

 i)Let $f_1,g_1,f_2,g_2 \in L^2(\Gamma_1)$ and $f_3,g_3,f_4,g_4 \in L^2(\R^3)$. The following anticommutation relations hold in the sense of quadratic form.
\begin{equation*}
 \begin{split}
   \{b_{1,+,\infty}(f_1),b_{1,+,\infty}^{*}(g_1)\}=& \langle f_1,g_1\rangle_{L^2(\Gamma_1)}\mathbf{1} \\
   \{b_{2, \epsilon,\infty}(f_2),\,b_{2, \epsilon',\infty}^*(g_2)\} =&\langle f_2,g_2\rangle_{L^2(\Gamma_1)}\delta_{\epsilon\epsilon'}\mathbf{1} \\
   \{b_{3,-,\infty}(f_3),b_{3,-,\infty}^{*}(g_3)\}=& \langle f_3,g_3\rangle_{L^2(\R^3)}\mathbf{1} \\
   \{b_{4,+,\infty}(f_4),b_{4,+,\infty}^{*}(g_4)\}=& \langle f_4,g_4\rangle_{L^2(\R^3)}\mathbf{1}\\
   \{b_{1,+,\infty}(f_1),b_{1,+,\infty}(g_1)\}=&\{b_{1,+,\infty}^*(f_1),b_{1,+,\infty}^{*}(g_1)\}= 0\\
  \{b_{1,+,\infty}(f_1),\,b_{2,\epsilon,\infty}^{\sharp}(f_2)\}=&\{b_{1,+,\infty}(f_1),b_{3,-,\infty}^{\sharp}(f_3)\}=0\\
  \{b_{1,+,\infty}(f_1),b_{4,+,\infty}^{\sharp}(f_4)\}=&0.\\
  \{b_{2, \epsilon,\infty}(f_2),\,b_{2, \epsilon',\infty}(g_2)\} =&\{b_{2, \epsilon,\infty}^*(f_2),\,b_{2, \epsilon',\infty}^*(g_2)\} =0\\
  \{b_{2, \epsilon,\infty}(f_2),\,b_{3,-,\infty}^\sharp(f_2)\} =&\{b_{2, \epsilon,\infty}(f_2),\,b_{4,+,\infty}^\sharp(f_4)\} =0\\
  \{b_{3,-,\infty}(f_3),b_{3,-,\infty}(g_3)\}=&\{b_{3,-,\infty}^*(f_3),b_{3,-,\infty}^{*}(g_3)\}=0\\
  \{b_{3,-,\infty}(f_3),b_{4,+,\infty}^\sharp(f_4)\}=&0\\
  \{b_{4,+,\infty}(f_4),b_{4,+,\infty}(g_4)\}=&\{b_{4,+,\infty}^*(f_4),b_{4,+,\infty}^{*}(g_4)\}=0.
 \end{split}
\end{equation*}
Here $\epsilon=\pm$.

ii)
\begin{equation*}
 \begin{split}
  e^{itH}b_{1,+,\infty}^{\sharp}(f_1)&=b_{1,+,\infty}^\sharp(e^{i\omega(\xi_1)t}f_1)e^{itH}.  \\
  e^{itH}b_{2,\pm,\infty}^{\sharp}(f_2)&=b_{2,\pm,\infty}^\sharp(e^{i\omega(\xi_2)t}f_2)e^{itH}.\\
  e^{itH}b_{3,-,\infty}^{\sharp}(f_3)&=b_{3,-,\infty}^\sharp(e^{i\omega(\xi_3)t}f_3)e^{itH}.\\
  e^{itH}b_{4,+,\infty}^{\sharp}(f_4)&=b_{4,+,\infty}^\sharp(e^{i\omega(\xi_4)t}f_4)e^{itH}.
   \end{split}
\end{equation*}

and the following pulltrough formulae are satisfied:

\begin{align*}
  [H,b_{1,+,\infty}^{*}(f_1)]= b_{1,+,\infty}^{*}(\omega(\xi_1)f_1), & \; [H,b_{1,+,\infty}(f_1)]=- b_{1,+,\infty}(\omega(\xi_1)f_1) \\
  [H,b_{2,\pm ,\infty}^{*}(f_2)]= b_{2,\pm ,\infty}^{*}(\omega(\xi_2)f_2), & \; [H,b_{2,\pm ,\infty}(f_2)]=- b_{2,\pm ,\infty}(\omega(\xi_2)f_2) \\
  [H,b_{3,-,\infty}^{*}(f_3)]= b_{3,-,\infty}^{*}(\omega(\xi_3)f_3), & \; [H,b_{3,-,\infty}(f_3)]=- b_{3,-,\infty}(\omega(\xi_3)f_3) \\
  [H,b_{4,+,\infty}^{*}(f_4)]= b_{4,+,\infty}^{*}(\omega(\xi_4)f_4), & \; [H,b_{4,+,\infty}(f_4)]=- b_{4,-,\infty}(\omega(\xi_4)f_4).
\end{align*}

iii)
\begin{equation*}
  b_{1,+,\infty}(f_1)\Omega_g=b_{2,\pm ,\infty}(f_2)\Omega_g=b_{3,-,\infty}(f_3)\Omega_g=b_{4,+,\infty}(f_4)\Omega_g=0
\end{equation*}
Here $\Omega_g$ is the ground state of H.

\end{proposition}

Our main result is the following theorem

\begin{theorem}\label{4.5}
Suppose Hypothesis 2.1-Hypothesis 4.2 and $g \leq g_1$. Then we have
\begin{equation*}
  \sigma_{ac}=[E,\infty).
\end{equation*}                                                                                                                                                       \end{theorem}

\begin{proof}

By \eqref{2.2} we have, for all sets of integers $(p,q,\bar q,r,s)$ in $\N^5$,
\begin{equation}\label{4.77}
 \mathfrak{F} = \bigoplus_{(p,q,\bar q,r,s)} \mathfrak{F}^{(p,q,\bar q, r,s)}.
\end{equation}
with
\begin{equation}\label{4.78}
 \mathfrak{F}^{(p,q,\bar q,r,s)}
 =(\otimes _a^{p} L^2(\Gamma_1))\otimes (\otimes _a^{q} L^2(\Gamma_1))\otimes
 (\otimes _a^{\bar q} L^2(\Gamma_1))\otimes
 (\otimes _a^{r} L^2(\R^3))\otimes
 (\otimes _a^{s} L^2(\R^3))\ .
\end{equation}
Here $p$ is the number of electrons, $q$ (resp. $\bar q$) is the number of muons
 (resp. antimuons), $r$
is the number of  antineutrinos $\overline{\nu_e}$ and $s$ is the number of neutrinos $\nu_\mu$.

Let $\{e^1_i|i=1,2,..\}$,$\{e^2_j|j=1,2,..\}$ and $\{f^2_k|k=1,2,..\}$ be tree orthonormal basis of $L^2(\Gamma_1)$. Let  $\{e^3_l|l=1,2,..\}$ and $\{e^4_m|m=1,2,..\}$ be two orthonormal basis of $L^2(\R^3)$.

Consider the following vectors of $\mathfrak{F}$
\begin{multline}\label{4.79}
  \produ_{1\leq\alpha\leq p}b_{1,+}^*(e^1_{i_\alpha})\prod_{1\leq\alpha\leq q}b_{2,+}^*(e^2_{j_\alpha})\prod_{1\leq\alpha\leq \bar q}b_{2,-}^*(f^2_{k_\alpha})
  \produ_{1\leq\alpha\leq r}b_{1,+}^*(e^3_{l_\alpha}) \prod_{1\leq\alpha\leq s}b_{4,+}^*(e^4_{m_\alpha})\Omega
\end{multline}

The indices are assumed ordered, $i_1<...<i_p$, $j_1<...<j_q$, $k_1<...<k_{\bar q}$, $l_1<...<l_r$ and  $m_1<...<m_s$.

The set, for $(p,q,\bar q,r,s)$ given in $\N^5$,
\begin{multline*}
  \mathbb{D}^{(p,q,\bar q,r,s)}= \{\Phi\in \mathfrak{F}^{(p,q,\bar q,r,s)}\;|\;\Phi \; \mbox{is a finite linear combination of basis vectors}\\                                     \mbox{of the form \eqref{4.79}} \}
\end{multline*}
is a dense domain in $ \mathfrak{F}^{(p,q,\bar q,r,s)}$.
The set of vectors of the form \eqref{4.79} is an orthonormal basis of  $\mathfrak{F}^{(p,q,\bar q,r,s)}$ (see \cite[Chapter 10]{Thaller1992}). Hence the vectors obtained in this way for $ p,q,\bar q,r,s,= 0,1,2,...$ form an orthonormal basis of $\mathfrak{F}$ and the set
\begin{multline*}
  \mathbb{D}= \{\Psi\in \mathfrak{F}\;|\;\Psi \; \mbox{is a finite linear combination of basis vectors}\\                                                                           \mbox{of the form \eqref{4.79} for $p,q,\bar q,r,s=0,1,2,\dotsb$} \}
\end{multline*}
is a dense domain in $ \mathfrak{F}$.

On the other hand we now introduce the following vectors of $\mathfrak{F}$

\begin{multline}\label{4.80}
  \produ_{1\leq\alpha\leq p}b_{1,+,\infty}^*(e^1_{i_\alpha})\prod_{1\leq\alpha\leq q}b_{2,+,\infty}^*(e^2_{j_\alpha})\prod_{1\leq\alpha\leq \bar q}b_{2,-,\infty}^*(f^2_{k_\alpha})\\
  \produ_{1\leq\alpha\leq r}b_{3,-,\infty}^*(e^3_{l_\alpha})\prod_{1\leq\alpha\leq s}b_{4,+,\infty}^*(e^4_{m_\alpha})
  \Omega_g
\end{multline}

 Let $\mathfrak{F}_\infty^{(p,q,\bar q,r,s)}$ denote the closed linear hull of vectors of the form \eqref{4.80}. It follows from proposition 4.4 that the set of vectors of the form \eqref{4.80} is an orthonormal basis of  $\mathfrak{F}_\infty ^{(p,q,\bar q,r,s)}$.

 The set, for $(p,q,\bar q,r,s)$ given in $\N^5$,
 \begin{multline*}
  \mathbb{D}_\infty^{(p,q,\bar q,r,s)}= \{\Phi\in \mathfrak{F}_\infty\;|\;\Phi \; \mbox{is a finite linear combination of basis vectors}\\                                          \mbox{of the form \eqref{4.80}}\}.
\end{multline*}
is a dense domain in $ \mathfrak{F}_\infty^{(p,q,\bar q,r,s)}$.

The asymptotic outgoing Fock pace denoted by $\mathfrak{F}_\infty$ is then defined by
\begin{equation}\label{4.81}
 \mathfrak{F}_\infty = \bigoplus_{p,q , \bar q, r, s} \mathfrak{F}_\infty^{(p,q,\bar q,r,s)}.
\end{equation}

The vectors of the form \eqref{4.80} obtained for $ p,q,\bar q,r,s,= 0,1,2,...$ form an orthonormal basis of $\mathfrak{F}$ and the set
\begin{multline*}
  \mathbb{D}_\infty= \{\Phi\in \mathfrak{F}_\infty\;|\;\Phi \; \mbox{is a finite linear combination of basis vectors}\\                                                             \mbox{of the form \eqref{4.80}\,for $p,q,\bar q,r,s= 0,1,2,\dotsb$}\}
\end{multline*}
is a dense domain in $ \mathfrak{F}_\infty$.

We now introduce the following linear operators, denoted by $W^{(p,q,\bar q,r,s)}_\infty$, and  defined on $\mathbb{D}^{(p,q,\bar q,r,s)}$ by
\begin{equation}\label{4.82}
\begin{split}
&W^{(p,q,\bar q,r,s)}_\infty \produ_{1\leq\alpha\leq p}b_{1,+}^*(e^1_{i_\alpha})\prod_{1\leq\alpha\leq q}b_{2,+}^*(e^2_{j_\alpha})\prod_{1\leq\alpha\leq \bar q}b_{2,-}^*(f^2_{k_\alpha})\produ_{1\leq\alpha\leq r}b_{1,+}^*(e^3_{l_\alpha})\\                                                                               &\quad\qquad\qquad\prod_{1\leq\alpha\leq s}b_{4,+}^*(e^4_{m_\alpha})\Omega  \\
&=\produ_{1\leq\alpha\leq p}b_{1,+,\infty}^*(e^1_{i_\alpha})\prod_{1\leq\alpha\leq q}b_{2,+,\infty}^*(e^2_{j_\alpha})\prod_{1\leq\alpha\leq \bar q}b_{2,-,\infty}^*(f^2_{k_\alpha})\produ_{1\leq\alpha\leq r}b_{3,-,\infty}^*(e^3_{l_\alpha})\\                                                                 &\quad\prod_{1\leq\alpha\leq s}b_{4,+,\infty}^*(e^4_{m_\alpha})\Omega_g.
\end{split}
\end{equation}

$W^{(p,q,\bar q,r,s)}_\infty$ can be uniquely extended to linear operators from $\mathbb{D}^{(p,q,\bar q,r,s)}$ to $\mathbb{D}_\infty^{(p,q,\bar q,r,s)}$. It then follows from prposition 4.4. that the operators $W^{(p,q,\bar q,r,s)}_\infty$ can be uniquely extended to unitary operators from $\mathbb{D}^{(p,q,\bar q,r,s)}$ to $\mathbb{D}_\infty^{(p,q,\bar q,r,s)}$

Let
\begin{equation}\label{4.83}
 W_\infty = \bigoplus_{p,q , \bar q, r, s}W^{(p,q,\bar q,r,s)}_\infty .
\end{equation}

Hence $W_\infty$ is a unitary operator from $ \mathfrak{F}$ to $ \mathfrak{F}_\infty$.

The operators $b_{1,+,\infty}(f_1)$,$b_{1,+,\infty}^*(g_1)$,$b_{2,+,\infty}(f_2)$, $b_{2,+,\infty}^*(g_2)$,$b_{2,-,\infty}(f_2)$, $b_{2,-,\infty}^*(g_2)$,$b_{3,-,\infty}(f_3)$, $b_{3,-,\infty}^*(g_3)$, $b_{4,+,\infty}(f_4)$ and $b_{4,+,\infty}^*(g_4)$ defined on $\mathfrak{F}_\infty $ generate a Fock representation of the ACR (see Proposition 4.4 i)).

By proposition 4.4 ii) we have

\begin{equation}\label{4.84}
  \begin{split}
&e^{itH}\produ_{1\leq\alpha\leq p}b_{1,+,\infty}^*(e^1_{i_\alpha})\prod_{1\leq\alpha\leq q}b_{2,+,\infty}^*(e^2_{j_\alpha})\prod_{1\leq\alpha\leq \bar q}b_{2,-,\infty}^*(f^2_{k_\alpha})\produ_{1\leq\alpha\leq r}b_{3,-,\infty}^*(e^3_{l_\alpha})\\                                                                               &\qquad \produ_{1\leq\alpha\leq s}b_{4,+,\infty}^*(e^4_{m_\alpha})\Omega_g \\
&=e^{iEt}\produ_{1\leq\alpha\leq p}b_{1,+,\infty}^*(e^{i\omega(\xi_1)t}e^1_{i_\alpha})\produ_{1\leq\alpha\leq q}b_{2,+,\infty}^*(e^{i\omega(\xi_2)t} e^2_{j_\alpha})\produ_{1\leq\alpha\leq \bar q}b_{2,-,\infty}^*(e^{i\omega(\xi_2)t}f^2_{k_\alpha})\\                                                            &\quad\qquad\produ_{1\leq\alpha\leq r}b_{3,-,\infty}^*(e^{i\omega(\xi_3)t}e^3_{l_\alpha})\produ_{1\leq\alpha\leq s}b_{4,+,\infty}^*(e^{i\omega(\xi_4)t}e^4_{m_\alpha})\Omega_g.
  \end{split}
\end{equation}

Hence $e^{iHt}$ leaves $\mathfrak{F}_\infty$ invariant and $H$ is both reduced by $\mathfrak{F}_\infty$ and $\mathfrak{F}_\infty^{\bot}$. Thus
\begin{equation*}
 \mathfrak{F}\simeq \mathfrak{F}_\infty\oplus\mathfrak{F}_\infty^{\bot}
\end{equation*}

In view of \eqref{4.5}, \eqref{4.48} and \eqref{4.84} we get


\begin{equation}\label{4.85}
  \begin{split}
&W_\infty e^{itH_0}\produ_{1\leq\alpha\leq p}b_{1,+}^*(e^1_{i_\alpha})\prod_{1\leq\alpha\leq q}b_{2,+}^*(e^2_{j_\alpha})\prod_{1\leq\alpha\leq \bar q}b_{2,-}^*(f^2_{k_\alpha})\produ_{1\leq\alpha\leq r}b_{3,-}^*(e^3_{l_\alpha})\\                                                                                             &\qquad \produ_{1\leq\alpha\leq s}b_{4,+}^*(e^4_{m_\alpha})\Omega \\
&=e^{i(H-E)t}W_\infty \produ_{1\leq\alpha\leq p}b_{1,+}^*(e^1_{i_\alpha})\produ_{1\leq\alpha\leq q}b_{2,+,}^*( e^2_{j_\alpha})\produ_{1\leq\alpha\leq \bar q}b_{2,-}^*(f^2_{k_\alpha})\\                                                                                                                   &\quad\qquad\produ_{1\leq\alpha\leq r}b_{3,-}^*(e^3_{l_\alpha})\produ_{1\leq\alpha\leq s}b_{4,+}^*(e^4_{m_\alpha})\Omega.
  \end{split}
\end{equation}

This yields
\begin{equation}\label{4.86}
  W_\infty e^{it(H_0+ E)}= =e^{iHt}W_\infty
\end{equation}
 Hence the reduction of $H$ to $\mathfrak{F}_\infty$ is unitarily equivalent to $H_0 + E$. Thus $ \sigma_{ac}(H)= [E,\infty)$ . This concludes the proof of theorem 4.5.
 \end{proof}

\section*{Acknowledgements.}
J.-C.G. acknowledges J.-M~ Barbaroux, J.~Faupin and G.~Hachem for helpful discussions.

\appendix

\section{The Dirac quantized fields for the electrons and the muons.}\label{Appendix}

\setcounter{equation}{0}

The appendices are based on the section 2 and section 3 of \cite{G17}. See also \cite{KB}, \cite{KBP1}, \cite{KBP2} and \cite{Hachem1993}.

$(s,n,p^{1},p^{3})$ are quantum variables of the electrons,the positrons, the muons and the antimuons in a uniform magnetic field. Here $s=\pm1$, $n\geq0$, $p^{1}\in\R$, $p^{3}\in\R$.

Let $\xi_{1}= (s,n,p^{1}_{e}, p^{3}_{e})$ be the quantum variables of a electron and let $\xi_{2}= (s,n,p^{1}_{\mu}, p^{3}_{\mu})$ be the quantum variables of a muon and of an antimuon.

We set $\Gamma_{1}=\{-1,1\}\times\N\times\R^2$ for the configuration space for the electrons, the muons and the antimuons. $L^2(\Gamma_{1})$ is the Hilbert space associated to each species of fermions.
 \begin{equation}\label{A.1}
   L^2(\Gamma_{1})=  l^2(L^2(\R^2))\oplus l^2(L^2(\R^2)
 \end{equation}

Let $\mathfrak{F}_{e}$ and $\mathfrak{F}_{\mu}$ denote the Fock spaces for the electrons and the muons respectively. Remark that $\mathfrak{F}_{\mu}$ is also the Fock space for the antimuons.

We have

\begin{equation}\label{A.2}
  \mathfrak{F}_{e} =\mathfrak{F}_{\mu}= \bigoplus_{n=0}^{\infty} \bigotimes_a^{n} L^2(\Gamma_{1}).
\end{equation}

$\bigotimes_a^n L^2(\Gamma_{1})$ is the antisymmetric n-th tensor power of $L^2(\Gamma_{1})$.

$\Omega_{\alpha}=(1,0,0,0,...)$ is the vacuum state in $\mathfrak{F}_{\alpha}$ for $\alpha= e, \mu .$

We shall use the notations
\begin{equation}\label{A.3}
  \begin{split}
     \int_{\Gamma_{1}}\d \xi_{1} &=\sum_{s=\pm1} \sum_{n\geq0}\int_{\R^2}\d p_{e}^1\d p_{e}^3\\
    \int_{\Gamma_{1}}\d \xi_{2} &=  \sum_{s=\pm1}\sum_{n\geq0}\int_{\R^2}\d p_{\mu}^1 \d p_{\mu}^3 .
  \end{split}
\end{equation}

$b_{+}(\xi_{1})$ (resp.$b_{+}^*(\xi_{1})$ is the annihilation (resp.creation)operator for the electron.

Let $\epsilon=\pm$.

$b_{\epsilon}(\xi_{2})$ (resp.$b_{\epsilon}^*(\xi_{2})$) are the annihilation (resp.creation)operators  for the muon when $\epsilon=+$ and for the antimuon when  $\epsilon=-$ .

The operators $b_{+}(\xi_{1})$, $b_{+}^*(\xi_{1})$, $b_{\epsilon}(\xi_{2})$ and $b_{\epsilon}^*(\xi_{2})$ fulfil the usual anticommutation relations (CAR)(see \cite{G17}and \cite{WI}).
Therefore the following anticommutation relations hold
\begin{equation}\label{A.4}
  \begin{split}
   &\{ b_{+}(\xi_1), b_{+}^*(\xi'_1)\} = \delta(\xi_1 - \xi'_1) \ ,\\
   &\{ b_{\epsilon}(\xi_2), b_{\epsilon'}^*(\xi'_2)\} =\delta_{\epsilon\epsilon'} \delta(\xi_2 - \xi'_2) \ ,\\
   &\{ b_{+}^{\sharp}(\xi_1), b_{\epsilon'}^{\sharp}(\xi_2)\} = 0  .
  \end{split}
\end{equation}
where $\{ b, b'\}= bb'+b'b$ and $b^{\sharp}= b$ or $b^*$.

In addition, following the convention described in \cite[Section 4.1]{WI} and
\cite[Section 4.2]{WI}, we assume that the fermionic creation and annihilation operators of different species of particles anticommute ( see \cite{BG} arXiv for explicit definitions). In our case this property will be verified by the creation and annihilation operators for the electrons, the muons, the antimuons, the antineutrinos associated to the electron and the neutrinos associated to the muons..

Recall that  for $\vp\in L^2(\Gamma_{1})$, the operators
\begin{equation}\label{A.5}
\begin{split}
  & b_{1,+}(\vp) = \int_{\Gamma_1} b(\xi_1) \overline{\vp(\xi_1)} \d \xi_1 . \\
  & b^*_{1,+}(\vp) = \int_{\Gamma_1} b^*(\xi_1) {\vp(\xi_1)} \d \xi_1.  \\
  & b_{2,\epsilon}(\vp) = \int_{\Gamma_1} b_{\epsilon}(\xi_2) \overline{\vp(\xi_2)} \d \xi_2 . \\
  & b^*_{2,\epsilon}(\vp) = \int_{\Gamma_1} b_{\epsilon}^*(\xi_2) {\vp(\xi_2)} \d \xi_2  .
\end{split}
\end{equation}
are bounded operators on $\mathfrak{F}_{(e)}$ and $\mathfrak{F}_{(\mu)}$  respectively satisfying

\begin{equation}\label{A.6}
  \begin{split}
  \| b^{\sharp}_{1,+}(\vp)\|&= \|\vp\|_{L^2}\ \\
  \| b^{\sharp}_{2,\epsilon}(\vp)\|&= \|\vp\|_{L^2}\
  \end{split}
\end{equation}

Set $\alpha=e,\mu$. The Dirac quantized fields for the electron and the muon, denoted by $\Psi_{(\alpha)}(\textbf{x})$, are given by

\begin{equation}\label{A.7}
  \begin{split}
   \Psi_{(\alpha)}(\mathbf{x}) = & \frac{1}{2\pi}
 \int\d \xi_1
 \Big(\mathrm{e}^{i(p_\alpha^1x^1+ p_\alpha^3x^3)} U^{(\alpha)}(x^2,\xi_\alpha)b_{+}(\xi_{\alpha})\\
 &\quad\quad\quad\quad\quad+ \mathrm{e}^{-i(p_\alpha^1x^1+ p_\alpha^3x^3)} W^{(\alpha)}(x^2,\xi_\alpha)b_{-}^*(\xi_\alpha)\Big) .
 \end{split}
\end{equation}
where $\xi_e=\xi_1$ and $\xi_\mu=\xi_2$. See \cite{G17}.

For  $\xi_{\alpha}= (s,n,p^{1}_{\alpha}, p^{3}_{\alpha})$ we have $ U^{(\alpha)}(x^2,\xi_\alpha)$=$U^{(\alpha)}_{s}(x^2,n,p_\alpha^1,p_\alpha^3)$.

For $s=1$ and $n\geq1$ $U^{(\alpha)}_{1}(x^2,n,p^1,p^3)$ is given by
\begin{equation}\label{A.8}
  U^{(\alpha)}_{+1}(x^2,n,p^1,p^3)=\biggl(\frac{E_n^{(\alpha)}(p^3)+m_\alpha}{2E_n^{(\alpha)}(p^3)}\biggr)^\frac{1}{2}
   \begin{pmatrix}
                                 I_{n-1}(\xi) \\
                                0 \\
                                \frac{p^3}{E_n^{(\alpha)}(p^3)+ m_\alpha}I_{n-1}(\xi)  \\
                                -\frac{\sqrt{2neB}}{E_n^{(\alpha)}(p^3)+ m_\alpha}I_n(\xi) \\
                              \end{pmatrix}
\end{equation}

where

\begin{equation}\label{A.9}
  \begin{split}
    \xi= & \sqrt{eB}(x^2-\frac{p^1}{eB}) \\
        I_n(\xi)=& \bigg(\frac{\sqrt{eB}}{n!2^n\sqrt{\pi}}\bigg)^\frac{1}{2}\exp(-\xi^2/2)H_n(\xi).
    \end{split}
\end{equation}
Here $H_n(\xi)$ is the Hermite polynomial of order n and we define
\begin{equation}\label{A.10}
  I_{-1}(\xi)=0
\end{equation}

For $n=0$ and $s=1$ we set
\begin{equation*}
  U^{(\alpha)}_{+1}(x^2,0,p^1,p^3)= 0
\end{equation*}

For $s=-1$ and  $n\geq0$ $U^{(\alpha)}_{-1}(x^2,n,p^1,p^3)$ is given by
\begin{equation}\label{A.11}
   U^{(\alpha)}_{-1}(x^2,n,p^1,p^3)=\biggl(\frac{E_n^{(\alpha)}(p^3)+m_\alpha}{2E_n^{(\alpha)}(p^3)}\biggr)^\frac{1}{2}
   \begin{pmatrix}
     0 \\
     I_n(\xi) \\
     -\frac{\sqrt{2neB}}{E_n^{(\alpha)}(p^3)+ m_\alpha} I_{n-1}(\xi) \\
     -\frac{p^3}{E_n^{(\alpha)}(p^3)+ m_\alpha}I_{n}(\xi) \\
   \end{pmatrix}
\end{equation}

Through out this work \underline{$e$ will be the positive unit} of charge taken to be equal to the proton charge.

$E_n^{(\alpha)}(p^3)$, $n\geq0$, is given by
\begin{equation}\label{A.12}
  E_n^{\alpha}(p^3)= \sqrt{m_\alpha^2+ (p^3)^2 +2neB}
\end{equation}

Note that
\begin{equation}\label{A.13}
  \int \d x^2 U^{(\alpha)}_{s}(x^2,n,p^1,p^3)^{\dagger} U^{(\alpha)}_{s'}(x^2,n,p^1,p^3)=\delta_{ss'}
\end{equation}
where $\dagger$ is the adjoint in $\C^4$.

In order to study the spectral theory of our Hamiltonian it is not necessary to know  $W^{(e)}(x^2,\xi_1)$ in \eqref{A.6}. We have to know $W^{(\mu)}(x^2,\xi_2)$ explicitly.

For $\xi_{2}= (s,n,p^{1}_{\mu}, p^{3}_{\mu})$ with $n>0$ we have
\begin{equation}\label{A.14}
  \begin{split}
     W^{(\mu)}(x^2,\xi_2)&=V^{(\mu)}_{-1}(x^2,n,-p^1_\mu,-p^3_\mu)\quad for\quad\xi_2=(1,n,p^1_\mu,p^3_\mu),n\geq0.\\
     W^{(\mu)}(x^2,\xi_2)&=V^{(\mu)}_{+1}(x^2,n,-p^1_\mu,-p^3_\mu)\quad for\quad\xi_2=(-1,n,p^1_\mu,p^3_\mu),n\geq1.\\
      W^{(\mu)}(x^2,\xi_2)&=0\quad for \quad \xi_1=(-1,0,p^1_\mu,p^3_\mu).
  \end{split}
\end{equation}

For $s=1$ and $n\geq1$  $V^{(\mu)}_{+1}(x^2,n,p^1,p^3)$ is given by

\begin{equation}\label{A.15}
   V^{(\mu)}_{+1}(x^2,n,p^1,p^3)=\biggl(\frac{E_n^{(\mu)}(p^3)+m_\mu}{2E_n^{(\mu)}(p^3)}\biggr)^\frac{1}{2}
   \begin{pmatrix}
                                -\frac{p^3}{E_n^{(\mu)}(p^3)+ m_\mu}I_{n-1}(\xi) \\
                                 \frac{\sqrt{2neB}}{E_n^{(\mu)}(p^3)+ m_\mu}I_n(\xi) \\
                                  I_{n-1}(\xi) \\
                                  0\\
                              \end{pmatrix}
\end{equation}

and for $n=0$ we set
\begin{equation*}
  V^{(\mu)}_{+1}(x^2,0,p^1,p^3)= 0
\end{equation*}

For $s=-1$ and  $n\geq0$ $V^{(\mu)}_{-1}(x^2,n,p^1,p^3)$ is given by

\begin{equation}\label{A.16}
   V^{(\mu)}_{-1}(x^2,n,p^1,p^3)=\biggl(\frac{E_n^{(\mu)}(p^3)+m_\mu}{2E_n^{(\mu)}(p^3)}\biggr)^\frac{1}{2}
   \begin{pmatrix}
     \frac{\sqrt{2neB}}{E_n^{(\mu)}(p^3)+ m_\mu} I_{n-1}(\xi) \\
      \frac{p^3}{E_n^{(\mu)}(p^3)+ m_\mu}I_{n}(\xi)\\
     0 \\
      I_n(\xi)  \\
   \end{pmatrix}
\end{equation}

Note that
\begin{equation}\label{A.17}
  \int \d x^2 V^{(\mu)}_{s}(x^2,n,p^1,p^3)^{\dagger} V^{(\mu)}_{s'}(x^2,n,p^1,p^3)=\delta_{ss'}
\end{equation}
where $\dagger$ is the adjoint in $\C^4$.\\

\section{The Dirac quantized fields for $\nu_{\mu}$ and $\overline{\nu_{e}}$.}\label{Appendix}
\setcounter{equation}{0}

We suppose that neutrinos and  antineutrinos are massless as in the Standard Model.

The quantum variables of the neutrinos and antineutrinos are the momenta and the helicities.

Let $\mathrm{\mathbf{P}}=(\mathrm{P}^{1},\mathrm{P}^{2},\mathrm{P}^{3})$ be the generators of space-translations.  $\mathrm{H}^{3}$ is the helicity operator $\frac{1}{2}\frac{\mathrm{\mathbf{P}}.\mathrm{\mathbf{\Sigma}}}{|\mathrm{\mathbf{P}}|}$ where $|\mathrm{\mathbf{P}}|=\big(\sqrt{\sum_{i=1}^{3}(\mathrm{P}^{i})^{2}}\big)$  and  $\mathrm{\mathbf{\Sigma}}=(\Sigma^1,\Sigma^2,\Sigma^3)$ with
for $j=1,2,3$
\begin{equation}\label{B.1}
   \Sigma^j =\left(
            \begin{array}{cc}
              \mathbf{\sigma_j} & 0 \\
              0 & \mathbf{\sigma_j} \\
            \end{array}
          \right)
\end{equation}

The helicity of the neutrino associated to the muon is $-\frac{1}{2}$. $\nu_{\mu}$ is left-handed. The helicity of the antineutrino associated to the electron  is $\frac{1}{2}$.$\overline{\nu}_e$ is right-handed.

Let $\xi_{3}= (\mathrm{\mathbf{p}},\frac{1}{2})$ be the quantum variables of the antineutrino $\overline{\nu}_e$  where $\mathrm{\mathbf{p}}\in\R^3$ is the momentum and $\frac{1}{2}$ is the helicity.Let $\xi_{4}= (\mathrm{\mathbf{p}},-\frac{1}{2})$ be the quantum variables of the neutrino $\nu_\mu$  where $\mathrm{\mathbf{p}}\in\R^3$ is the momentum and $-\frac{1}{2}$ is the helicity.

$L^2 (\R^3)$ is the Hilbert space of the states of the neutrinos $\nu_\mu$ and of the antineutrinos $\overline{\nu}_e$.Let $\mathfrak{F}_{(\nu_\mu)}$ and $\mathfrak{F}_{(\overline{\nu}_e)}$ denote the Fock spaces for the neutrinos and the antineutrinos respectively.

We have

\begin{equation}\label{B.2}
  \mathfrak{F}_{\nu_\mu}=\mathfrak{F}_{\overline{\nu}_e}=\bigoplus_{n=0}^\infty \bigotimes_a^nL^2(\R^3) \end{equation}


$\Omega_{\beta}=(1,0,0,0,...)$ is the vacuum state in $\mathfrak{F}_{\beta}$ for $\beta=\overline{\nu_e}, \mu_\mu .$

In the sequel we shall use the notations
\begin{equation}\label{B.3}
  \begin{split}
    \int_{\R^3}\d \xi_{3} &= \int_{\R^3}\d^3\mathrm{\mathbf{p}} \\
    \int_{\R^3}\d \xi_{4} &= \int_{\R^3}\d^3\mathrm{\mathbf{p}}
  \end{split}
\end{equation}

$b_{-}(\xi_{3})$ and $b_{-}^*(\xi_{3})$) are the annihilation and creation operators for the antineutrino associated to the electron respectively .
$b_{+}(\xi_{4})$ and $b_{+}^*(\xi_{4})$) are the annihilation and creation operators for the neutrino associated to the muon respectively.
The operators $b_{-}^{\sharp}(\xi_{3})$ and $b_{+}^\sharp(\xi_{4})$, fulfil the usual anticommutation relations (CAR) and they anticommute with $b_{+}^{\sharp}(\xi_{1})$ and
$b_{\epsilon}^{\sharp}(\xi_{2})$ according to the convention described in \cite[Section 4.1]{WI}. See \cite{BG} arXiv for explicit definitions.

Therefore the following anticommutation relations hold
\begin{equation}\label{B.4}
  \begin{split}
   &\{ b_{-}(\xi_3), b_{-}^*(\xi'_3)\} = \delta(\xi_3 - \xi'_3) \ ,\\
   &\{ b_{+}(\xi_4),  b_{+}(\xi'_4)\} = \delta(\xi_4 - \xi'_4) \ ,\\
   &\{ b_{-}^{\sharp}(\xi_3), b_{+}^{\sharp}(\xi_4)\}=0 \\
   &\{ b_{-}^{\sharp}(\xi_3), b_{\epsilon}^{\sharp}(\xi_2)\} = \{ b_{-}^{\sharp}(\xi_3), b_{+}^{\sharp}(\xi_1)\} = 0.\\
   &\{ b_{+}^{\sharp}(\xi_4), b_{\epsilon}^{\sharp}(\xi_2)\} = \{ b_{+}^{\sharp}(\xi_4), b_{+}^{\sharp}(\xi_1)\} = 0.
  \end{split}
\end{equation}

Recall that, for $\vp\in L^2(\R^3)$, the operators
\begin{equation}\label{B.5}
\begin{split}
  & b_{4,+}(\vp) = \int_{\R^3} b_{+}(\xi_4) \overline{\vp(\xi_4)} \d \xi_4 . \\
  & b_{3,-}(\vp) = \int_{\R^3} b_{-}(\xi_3)\overline{\vp(\xi_3)} \d \xi_3 . \\
  & b^*_{4,+}(\vp) = \int_{\R^3} b_{+}^*(\xi_4) {\vp(\xi_4)} \d \xi_4 .\\
  & b^*_{3,-}(\vp) = \int_{\R^3} b_{-}^*(\xi_3) {\vp(\xi_3)} \d \xi_3 .
\end{split}
\end{equation}
are bounded operators on $\mathfrak{F}_{(\nu_\mu)}$ and $\mathfrak{F}_{(\overline{\nu_e})}$ respectively  satisfying
\begin{equation}\label{B.6}
  \| b^\sharp_{4}(\vp)\|  =\| b^\sharp_{3}(\vp)\|= \|\vp\|_{L^2}.
\end{equation}

The Dirac quantized fields for the neutrinos and antineutrinos associated to the electron and the muon respectively are denoted by $\Psi_{(\nu_e)}(\mathbf{x})$ and  $\Psi_{(\nu_\mu)}(\mathbf{x})$.

We have

\begin{equation}\label{B.7}
\begin{split}
   \Psi_{(\nu_e)}(\mathbf{x}) = & (\frac{1}{2\pi})^\frac{3}{2}\Bigl(
 \int\d\widetilde{\xi}_3
 e^{i(\mathrm{\mathbf{p}}.\mathrm{\mathbf{x}})} U^{(\nu_e)}(\widetilde{\xi}_3)b_{+}(\widetilde{\xi}_3)\\
 &\qquad\qquad + \int\d\xi_3 e^{-i(\mathrm{\mathbf{p}}.\mathrm{\mathbf{x}})}W^{(\overline{\nu}_e)}(\xi_3)b_{-}^*(\xi_3)\Bigr) .
 \end{split}
\end{equation}
and
\begin{equation}\label{B.8}
\begin{split}
   \Psi_{(\nu_\mu)}(\mathbf{x}) = & (\frac{1}{2\pi})^\frac{3}{2}\Bigl(
 \int\d\xi_4
 e^{i(\mathrm{\mathbf{p}}.\mathrm{\mathbf{x}})} U^{(\nu_\mu)}(\xi_4)b_{+}(\xi_4)\\
 &\qquad\qquad + \int\d\widetilde{\xi}_4 e^{-i(\mathrm{\mathbf{p}}.\mathrm{\mathbf{x}})}W^{(\overline{\nu_\mu})}(\widetilde{\xi}_4)b_{-}^*(\widetilde{\xi}_4)\Bigr) .
 \end{split}
\end{equation}
where $\widetilde{\xi}_3=(\mathrm{\mathbf{p}},-\frac{1}{2})$ and $\widetilde{\xi}_4=(\mathrm{\mathbf{p}},\frac{1}{2})$ with, for $\beta=3,4$,
\begin{equation}\label{B.9}
  \begin{split}
    \int_{\R^3}\d \xi_{\beta} &= \int_{\R^3}\d^3p \\
    \int_{\R^3}\d \widetilde{\xi}_{\beta} &= \int_{\R^3}\d^3p
  \end{split}
\end{equation}
See \cite{DG}.

For the purpose of this paper one only needs to know $W^{(\overline{\nu_e})}(\xi_3)$ and $U^{(\nu_\mu)}(\xi_4)$ explicitly. $ U^{(\nu_e)}(\widetilde{\xi}_3)$ and $W^{(\overline{\nu}_\mu)}(\widetilde{\xi}_4)$ are given in \cite{G17}.

By \cite[(3.6), (3.7), (3.24), (3.32)]{G17}and \cite{Thaller1992} we have

\begin{equation}\label{B.10}
   U^{(\nu_\mu)}(\xi_{4})= U^{(\nu_\mu)}(\mathrm{\mathbf{p}},-\frac{1}{2})=)=\frac{1}{\sqrt{2}}\begin{pmatrix}
                                                                                h_-(\mathrm{\mathbf{p}}) \\
                                                                                - h_-(\mathrm{\mathbf{p}}) \\
                                                                                             \end{pmatrix}
\end{equation}
with
\begin{equation}\label{B.11}
  h_-(\mathrm{\mathbf{p}})=\frac{1}{\sqrt{2|\mathrm{\mathbf{p}}|(|\mathrm{\mathbf{p}}|-p^3)}}\begin{pmatrix}
                                                                                        p^3- |\mathrm{\mathbf{p}}|\\
                                                                                       p^1 + ip^2  \\
                                                                                      \end{pmatrix}
\end{equation}

and for $|\mathrm{\mathbf{p}}|= p^3$ we set
\begin{equation*}
  h_-(\mathrm{\mathbf{p}})=\begin{pmatrix}
                             0 \\
                             1 \\
                           \end{pmatrix}
\end{equation*}

Moreover we have

\begin{equation}\label{B.12}
   W^{(\nu_e)}(\xi_{3})=V^{(\nu_e)}(-\mathrm{\mathbf{p}},\frac{1}{2})=\frac{1}{\sqrt{2}}\begin{pmatrix}
                                                                           - h_+(\mathrm{\mathbf{-p}}) \\
                                                                           h_+(\mathrm{\mathbf{-p}}) \\
                                                                               \end{pmatrix}
\end{equation}
with
\begin{equation}\label{B.13}
  h_+(\mathrm{\mathbf{-p}})=\frac{1}{\sqrt{2|\mathrm{\mathbf{p}}|(|\mathrm{\mathbf{p}}|+p^3)}}\begin{pmatrix}
                                                                                        -p^1+ip^2 \\
                                                                                       |\mathrm{\mathbf{p}}|+p^3  \\
                                                                                      \end{pmatrix}
\end{equation}
and for $|\mathrm{\mathbf{p}}|= -p^3$ we set
\begin{equation*}
  h_+(\mathrm{\mathbf{-p}})=\begin{pmatrix}
                             1 \\
                             0 \\
                           \end{pmatrix}
\end{equation*}

Note that
\begin{equation}\label{B.14}
  \| U^{(\nu_\mu)}(\xi_{4})\|_{\C^4}=\| W^{(\overline{\nu_e})}(\xi_{3})\|_{\C^4}=1
\end{equation}

\end{document}